\documentclass[12pt]{article}

\usepackage{amsmath,amsthm}
\usepackage{amssymb}
\usepackage{color}
\usepackage{breqn}
\usepackage{enumerate}
\usepackage{graphicx}
\usepackage{bm}
\usepackage{bbm}
\usepackage[affil-it]{authblk}
\usepackage{tabu}
\usepackage{bold-extra}
\usepackage[hang,flushmargin]{footmisc} 
\usepackage[hypertex]{hyperref}
\usepackage{mathtools}
 \usepackage{relsize}
 
\newtheorem{theorem}{Theorem}

\newtheorem{lemma}[theorem]{Lemma}
\newtheorem{proposition}[theorem]{Proposition}

\newtheorem{observation}[theorem]{Observation}
\theoremstyle{definition}
\newtheorem{defin}[theorem]{Definition}

\newcounter{boldSectionCounter}
\stepcounter{boldSectionCounter}
\newcounter{boldSubsectionCounter}
\stepcounter{boldSubsectionCounter}

\newcommand{\boldSubsection}[1]{
	 \addtocounter{boldSectionCounter}{-1}
   \noindent {\bfseries{\scshape \arabic{boldSectionCounter}.\arabic{boldSubsectionCounter}. #1}}\\[6pt]
   \stepcounter{boldSubsectionCounter}
   \addtocounter{boldSectionCounter}{1}
}

\newcommand{\tdt}{
	\times\dots\times
}

\newcommand{\boldSection}[1]{
   \large\begin{center}\noindent {\bfseries{\scshape \S \arabic{boldSectionCounter} #1}}\\[6pt]\end{center}\normalsize
   \stepcounter{boldSectionCounter}
	 \setcounter{boldSubsectionCounter}{1}
}	

\newcommand{\spn}{\operatorname{span}}

\newcommand{\tightoverset}[2]{
  \mathop{#2}\limits^{\vbox to -.5ex{\kern-1.15ex\hbox{$#1$}\vss}}}

\newcommand{\bigconv}[1]{
	\mathbf{C}_{#1}
}

\newcommand\restr[2]{{% we make the whole thing an ordinary symbol
  \left.\kern-\nulldelimiterspace % automatically resize the bar with \right
  #1 % the function
  \vphantom{\big|} % pretend it's a little taller at normal size
  \right|_{#2} % this is the delimiter
}}	

\newcommand\blfootnote[1]{%
  \begingroup
  \renewcommand\thefootnote{}\footnote{#1}%
  \addtocounter{footnote}{-1}%
  \endgroup
}

\newcommand\ex{
	\mathop{\mathbb{E}}
}

\makeatletter
\newcommand*\bcdot{\mathpalette\bigcdot@{0.5}}
\newcommand*\bigcdot@[2]{\mathbin{\vcenter{\hbox{\scalebox{#2}{$\m@th#1\bullet$}}}}}
\makeatother

\makeatletter
\def\blfootnote{\gdef\@thefnmark{}\@footnotetext}
\makeatother	
	
\begin{document}
\begin{center}\Large\noindent{\bfseries{\scshape Polynomial bound for partition rank in terms of analytic rank}}\\[24pt]
\normalsize\noindent{\scshape Luka Mili\'cevi\'c\dag}\\[6pt]
\end{center}
\blfootnote{\noindent\dag\ Mathematical Institute of the Serbian Academy of Sciences and Arts, Belgrade, Serbia.\\\phantom{\dag\ }Email: luka.milicevic@turing.mi.sanu.ac.rs}
%\begin{flushright}Date: 14 September 2018\end{flushright}

\footnotesize
\begin{changemargin}{1in}{1in}
\centerline{\sc{\textbf{Abstract}}}
\hspace{12pt}~Let $G_1, \dots, G_k$ be vector spaces over a finite field $\mathbb{F} = \mathbb{F}_q$ with a non-trivial additive character $\chi$. The analytic rank of a multilinear form $\alpha \colon G_1 \tdt G_k \to \mathbb{F}$ is defined as $\operatorname{arank}(\alpha) = -\log_q \ex_{x_1 \in G_1, \dots, x_k\in G_k} \chi\big(\alpha(x_1,\dots, x_k)\big)$. The partition rank $\operatorname{prank}(\alpha)$ of $\alpha$ is the smallest number of maps of partition rank 1 that add up to $\alpha$, where a map is of partition rank 1 if it can be written as a product of two multilinear forms, depending on different coordinates. It is easy to see that $\operatorname{arank}(\alpha) \leq O\Big(\operatorname{prank}(\alpha)\Big)$ and it has been known that $\operatorname{prank}(\alpha)$ can be bounded from above in terms of $\operatorname{arank}(\alpha)$. In this paper, we improve the latter bound to polynomial, i.e.\ we show that there are quantities $C, D$ depending on $k$ only such that $\operatorname{prank}(\alpha) \leq C (\operatorname{arank}(\alpha)^D + 1)$. As a consequence, we prove a conjecture of Kazhdan and Ziegler.\\
\phantom{aaa}The same result was obtained independently and simultaneously by Janzer. 
\end{changemargin}

\boldSection{Introduction}
Throughout the paper $G_1, \dots, G_k$ are vector spaces over a finite field $\mathbb{F}$ and $\chi$ is a non-trivial additive character of $\mathbb{F}$. In order to measure how much the distribution of values of a multilinear form $\alpha \colon G_1 \tdt G_k \to \mathbb{F}$ deviates from the uniform distribution Gowers and Wolf~\cite{GowersWolf} introduced the notion of \emph{analytic rank}.

\begin{defin}[Analytic rank]\emph{Analytic rank} of $\alpha$ is defined as 
\[\operatorname{arank}(\alpha) = -\log_{|\mathbb{F}|} \ex_{x_1\in G_1, \dots, x_k\in G_k} \chi\Big(\alpha(x_1, \dots, x_k)\Big).\]\end{defin}

Note that this definition does not depend on the choice of $\chi$ and that for the case of two variables coincides with the usual algebraic rank of the corresponding matrix. Another way to generalize the notion of algebraic rank to multilinear forms is the \emph{partition rank}, introduced by Naslund~\cite{Naslund}, which is defined as follows.

\begin{defin}[Partition rank]Let $\alpha \colon G_1 \tdt G_k \to \mathbb{F}$ be a multilinear form. We say that $\alpha$ is of partition rank 1 if there is a set $\emptyset \not= I \subset [k]$ and multlinear forms $\beta \colon \prod_{i \in I} G_i \to \mathbb{F}$ and $\gamma \colon \prod_{i \in [k]\setminus I} G_i \to \mathbb{F}$ such that
\[\Big(\forall x_1 \in G_1, \dots, x_k \in G_k\Big) \alpha(x_1, \dots, x_k) = \beta(x_i \colon i \in I) \gamma(x_i \colon i \in [k] \setminus I).\]
In general, \emph{partition rank} $\operatorname{prank}(\alpha)$ of $\alpha$ is the smallest integer $r$ such that $\alpha$ is a sum of $r$ multilinear forms of partition rank 1. We also set $\operatorname{prank}(0) = 0$.\end{defin}

It is easy to see that $\operatorname{arank}(\alpha) \leq O\Big(\operatorname{prank}(\alpha)\Big)$ and Lovett~\cite{Lov} proved $\operatorname{arank}(\alpha) \leq \operatorname{prank}(\alpha)$. On the other hand, it turns out that we can bound partition rank in terms of analytic rank. This was first proved by Bhowmick and Lovett~\cite{BhowLov}, and they showed that $\operatorname{prank}(\alpha) \leq f(\operatorname{arank}(\alpha), k, |\mathbb{F}|)$, where $f$ has an Ackerman-type dependence on its parameters. This bound was recently improved significantly by Janzer~\cite{Janzer1} to tower of exponentials of height depending on $k$. In this paper we prove 
\begin{theorem}\label{strongIntro}For given $k$, there are constants $C, D$ such that $\operatorname{prank}(\alpha) \leq C (\operatorname{arank}(\alpha)^D + 1).$\end{theorem}
We refer to this result as the \emph{strong inverse theorem for multilinear forms of low analytic rank} or simply as \emph{strong inverse theorem}.\footnote{We emphasize the word \emph{strong} to make distiction from a weak version of the inverse theorem which we also prove in this paper.} The constants $C,D$ can be taken to be $C = 2^{k^{2^{O(k^2)}}}, D = 2^{2^{O(k^2)}}$. Note also that, since every non-zero form $\alpha \colon G_1 \tdt G_k\to \mathbb{F}$ takes at least $\Big(\frac{|\mathbb{F}| - 1}{|\mathbb{F}|}\Big)^k |G_1|\cdots |G_k|$ non-zero values, there is a constant $a_0$, depending on $k$ and $\mathbb{F}$, such that for all non-zero forms $\alpha$, $\operatorname{arank}(\alpha) \geq a_0$. Hence, we may rewrite the bound in Theorem~\ref{strongIntro} as 
\[\operatorname{prank}(\alpha) \leq C' \operatorname{arank}(\alpha)^D,\]
but the new constant $C'$ would need to depend on $\mathbb{F}$ as well.\\[3pt]
\indent Theorem~\ref{strongIntro} was also proved by Janzer~\cite{Janzer2} independently and simultaneously, using a different argument.\\

Before the study of ranks of multilinear forms, an important topic of study has been the distribution of multivariate polynomials over finite fields. In this direction, we have the following result proved by Green and Tao~\cite{GreenTao} (with the restriction on the degree of the polynomial to be bounded by the size of the field), and Kaufman and Lovett~\cite{KaufLov} (without the restriction on the degree), and was applied by Bhowmick and Lovett~\cite{BhowLov} to coding theory and effective algebraic geometry.

\begin{theorem}[Green and Tao~\cite{GreenTao}; Kaufman and Lovett~\cite{KaufLov}]\label{biasedPoly} Let $\mathbb{F}$ be a field of prime order. Suppose that $f \colon \mathbb{F}^n \to \mathbb{F}$ is a polynomial of degree $d$ ($d < |\mathbb{F}|$ in the result of Green and Tao). If $\Big|\ex_{x_1, \dots, x_n \in \mathbb{F}} \chi\Big(f(x_1, \dots, x_n)\Big)\Big| \geq c$, then there is $r \leq O_{p, d, c^{-1}}(1)$, polynomials $f_1, \dots, f_r$ of degree $d-1$, and a map $g \colon \mathbb{F}^r \to \mathbb{F}$ such that
\[\Big(\forall x_1, \dots, x_n \in \mathbb{F}\Big) f(x_1, \dots, x_n) = g\Big(f_1(x_1, \dots, x_n), \dots, f_r(x_1, \dots, x_n)\Big).\]\end{theorem}

As observed by Green and Tao in~\cite{GreenTao} and by Janzer in~\cite{Janzer1}, it is easy to deduce this result from the strong inverse theorem, at least when $d < \operatorname{char} \mathbb{F}$. We include a very short sketch to confirm that we may take polynomial bounds in Theorem~\ref{biasedPoly} as well. This proves a conjecture of Kazhdan and Zielger (Conjecture 3.5 in~\cite{KazhZie}).

\begin{proof}[Sketch proof of Theorem~\ref{biasedPoly}]Consider a symmetric multilinear form $\alpha \colon (\mathbb{F}^n)^d \to \mathbb{F}$ such that $f(x) = \alpha(x, x, \dots, x)$. Since $d < \operatorname{char} \mathbb{F}$, $d!$ is invertible in $\mathbb{F}$, and such a form exists. Notice that for all $x, y^1, \dots, y^{d-1} \in \mathbb{F}^n$
\[\begin{split}\sum_{S \subset [d-1]} (-1)^{d-1-|S|}f\Big(x + \sum_{i \in S} y^i\Big) = (d-1)!\alpha(x, y^1, \dots, y^{d-1}) + &g(y^1, \dots, y^{d-1})\\
 + &\sum_{i \in [d-1]} h_i(y^1, \dots, y^{i-1}, x, y^{i+1}, \dots, y^{d-1}),\end{split}\]
for some maps $g, h_1, \dots, h_{d-1}$. Observe also that
\[\begin{split}\Big|\ex_{x\in \mathbb{F}^n} \chi\Big(\alpha(x, \dots, x)\Big)\Big|^2 = \Big|\ex_{x, y\in \mathbb{F}^n} &\chi\Big(\alpha(x + y, \dots, x + y)\Big)\Big|^2 \leq \ex_{x\in \mathbb{F}^n} \Big| \ex_{y\in \mathbb{F}^n} \chi\Big(\alpha(x + y, \dots, x+ y)\Big)\Big|^2 \\
= &\ex_{x,y,z\in \mathbb{F}^n} \chi\Big(\alpha(x + y, \dots, x+ y) - \alpha(x + z, \dots, x + z)\Big)\\
= &\ex_{x,y\in \mathbb{F}^n} \chi\Big(\alpha(x + y, \dots, x+y) - \alpha(x, \dots, x)\Big).\end{split}\]
Applying this $d-1$ times in total, we get
\[\begin{split}c^{2^{d-1}} = &\Big|\ex_{x\in \mathbb{F}^n} \chi(f(x))\Big|^{2^{d-1}} = \Big|\ex_{x\in \mathbb{F}^n} \chi\Big(\alpha(x, \dots, x)\Big)\Big|^{2^{d-1}} \leq \ex_{x, y^1, \dots, y^{d-1}\in \mathbb{F}^n} \chi\Big(\sum_{S \subset [d-1]} (-1)^{d-1-|S|}f\big(x + \sum_{i \in S} y^i\big)\Big)\\
 = &\ex_{x, y^1, \dots, y^{d-1}\in \mathbb{F}^n} \chi\Big((d-1)!\hspace{2pt}\alpha(x, y^1, \dots, y^{d-1}) + g(y^1, \dots, y^{d-1}) + \sum_{i \in [d-1]} h_i(y^1, \dots, y^{i-1}, x, y^{i+1}, \dots, y^{d-1})\Big).\end{split}\]
Hence, by Cauchy-Schwarz inequality for the box norm (Lemma~\ref{unifBound}), we get
\[\ex_{x, y^1, \dots, y^{d-1}\in \mathbb{F}^n} \chi\Big(\alpha(x, y^1, \dots, y^{d-1}\Big) \geq c^{2^{2d-2}}.\]
Apply the strong inverse theorem to find $r \leq C \Big(2^{2d-2} \log_{|\mathbb{F}|} (|\mathbb{F}|c^{-1})\Big)^D$, sets $\emptyset \not= I_i \subsetneq [d]$ and multilinear forms $\beta_i \colon \prod_{j \in I_i} G_j \to \mathbb{F}, \gamma_i \colon \prod_{j \in [d] \setminus I_i} G_j \to \mathbb{F}$ for $i \in [r]$ such that
\[\alpha(u^1, \dots, u^d) = \sum_{i \in [r]} \beta_i(u_j \colon j \in I_i) \gamma_i(u_j \colon j \in [d] \setminus I_i).\]
Then $f(x) = \alpha(x, \dots, x) = \sum_{i \in [r]} \beta_i(x, \dots, x) \gamma_i(x, \dots, x)$, as desired.\end{proof}

The proof of the strong inverse theorem given here is entirely self-contained and in particular does not use other results of additive combinatorics such as Freiman's theorem, and, as it is clear from the bounds we obtain here, we do not apply a regularity lemma. In fact, it is probable that the results of this paper may replace the use of regularity lemmas in similarly algebraic settings. In the next subsection, we list main results of this paper, discuss the proofs and the ideas in more detail.\\
\pagebreak

\boldSubsection{Main results and overview of the argument}
\textbf{Notation.}\ In the rest of the paper, we use the following convention to save writing in situations where we have many indices appearing in predictable patterns. Instead of whole sequence $x_1, \dots, x_m$, we write $x_{[m]}$, and we write $x_I$ for $I \subset [m]$ to be the subsequence with indices in $I$. This applies to products as well: $G_{[k]}$ stands for $\prod_{i \in [k]} G_i$ and $G_I = \prod_{i \in I} G_i$. For example, instead of writing $\alpha \colon \prod_{i \in I} G_i \to \mathbb{F}$ and $\alpha(x_i \colon i \in I)$, we write $\alpha \colon G_I \to \mathbb{F}$ and $\alpha(x_I)$.\\
\indent Given $\mathbb{F}$-vector spaces $G_1, \dots, G_k, H$, a map $\alpha \colon G_{[k]} \to H$ is said to be \emph{multilinear} if it is linear in each coordinate, that is, whenever $x_{[i-1]} \in G_{[i-1]}, x_{[i+1, k]} \in G_{[i+1, k]}$ and $y,z \in G_i$, then 
\[\alpha(x_{[i-1]}, y + z, x_{[i+1, k]}) = \alpha(x_{[i-1]}, y, x_{[i+1, k]}) + \alpha(x_{[i-1]}, z, x_{[i+1, k]}).\]
Similarly, $\alpha$ is \emph{multiaffine} if it is affine in each coordinate, i.e.\ whenever $x_{[i-1]} \in G_{[i-1]}, x_{[i+1, k]} \in G_{[i+1, k]}$ and $y,z, w \in G_i$, then 
\[\alpha(x_{[i-1]}, y + z - w, x_{[i+1, k]}) = \alpha(x_{[i-1]}, y, x_{[i+1, k]}) + \alpha(x_{[i-1]}, z, x_{[i+1, k]}) - \alpha(x_{[i-1]}, w, x_{[i+1, k]}).\]
Also, we refer to the zero set of a multiaffine map $\alpha \colon G_{[k]} \to H$, where $H$ is a vector space over $\mathbb{F}$, as \emph{variety}, and the codimension of a variety is $\dim H$. Another convention we adopt is that we write $\ex_x$, without specifying the set from which $x$ is taken, when this causes no confusion. Frequently we shall consider `slices' of sets $S\subset G_{[k]}$, by which we mean sets $S_{x_I} = \{y_{[k]\setminus I} \in G_{[k] \setminus I} \colon (x_I, y_{[k] \setminus I}) \in S\}$, for $I \subset [k], x_I \in G_I$. Occasionally, we might have a single element $z \in G_i$ instead of $x_I$, and in this case we write $S_{i \colon z}$ for the resulting slice, since the direction $i$ is not clear from the notation $z$, unlike in the case of $x_I$. Finally, for each vector space $G_i$, fix a dot product $\cdot$. We need this for the characterization of linear forms on $G_i$ -- each linear form $\phi \colon G_i \to \mathbb{F}$ takes form $\phi(x) = x \cdot u$ for an element $u \in G_i$.\\[6pt]
\noindent\textbf{Results and outline.}\ Our first main result is the weak version of the (strong) inverse theorem.
 
\begin{theorem}[Weak inverse theorem for maps of low analytic rank - \textbf{Weak}($k$)]\label{weakInvThm}For given $k$, there are constants $C = C^{\bm{weak}}_k, D = D^{\bm{weak}}_k > 0$ with the following property. Suppose that $\alpha \colon G_{[k]} \to \mathbb{F}$ is a multilinear form such that $\ex_{x_{[k]}} \chi\Big(\alpha(x_{[k]})\Big) \geq c$, for some $c > 0$. Then, there is $r \leq C \log_{|\mathbb{F}|}^D (|\mathbb{F}|c^{-1})$ and there are multilinear maps $\beta_i \colon G_{I_i} \to \mathbb{F}$, $i \in [r]$, where $\emptyset \not= I_i \subset [k-1]$ such that 
\[\Big\{x_{[k]} \in G_{[k]} \colon (\forall i \in [r]) \beta_i(x_{I_i}) = 0\Big\} \subset \Big\{x_{[k]} \in G_{[k]} \colon \alpha(x_{[k]}) = 0\Big\}.\]
\end{theorem}

Note that there is a multilinear map $A \colon G_{[k-1]} \to G_k$ such that for each $x_{[k]} \in G_{[k]}$, $\alpha(x_{[k]}) = A(x_{[k-1]}) \cdot x_k$. Then $\Big|\Big\{x_{[k-1]} \in G_{[k-1]} \colon A(x_{[k-1]}) = 0\Big\}\Big| = \Big(\ex_{x_{[k]}} \chi(\alpha(x_{[k]}))\Big)  |G_{[k-1]}|$. Thus, another way to phrase the weak inverse theorem is to say that every dense variety contains a low-codimensional variety. On the other hand, it is very easy to see that low-codimensional varieties are necessarily dense, see Lemma~\ref{basicDensityL}.\\
\indent Next, we have the strong inverse theorem.

\begin{theorem}[Strong inverse theorem for maps of low analytic rank - \textbf{Strong}($k$)]\label{strongInvThm}For given $k$, there are constants $C = C^{\bm{strong}}_k, D = D^{\bm{strong}}_k > 0$ with the following property. Suppose that $\alpha \colon G_{[k]} \to \mathbb{F}$ is a multilinear form such that $\ex_{x_{[k]}} \chi(\alpha(x_{[k]})) \geq c$, for some $c > 0$. Then, there is $r \leq C \log_{|\mathbb{F}|}^D (|\mathbb{F}|c^{-1})$ and there are multilinear maps $\beta_i \colon G_{I_i} \to \mathbb{F}$ and $\gamma_i \colon G_{[k] \setminus I_i} \to \mathbb{F}$, $i \in [r]$, where $\emptyset \not= I_i \subset [k-1]$ such that 
\[\Big(\forall x_{[k]} \in G_{[k]}\Big)\hspace{6pt}\alpha(x_{[k]}) = \sum_{i \in [r]} \beta_i(x_{I_i}) \gamma_i(x_{[k] \setminus I_i}).\]
\end{theorem}

We remark that the bounds on the constants claimed in the introduction, namely $C^{\bm{strong}}_k = 2^{k^{2^{O(k^2)}}}$ and $D^{\bm{strong}}_k = 2^{2^{O(k^2)}}$, follow from inequalities~\eqref{weakBounds} and~\eqref{strongBounds} which appear later in the paper. In fact, bounds of the same form hold in all results stated in this subsection.\\

%\begin{theorem}[Inner approximation of Bohr varieties - \textbf{Inner}($k$)]\label{innerAppThm}For given $k$, there are constants $C=C^{\bm{inner}}_k, D=D^{\bm{inner}}_k > 0$ with the following property. Let $\epsilon > 0$ and let $B \colon G_{[k]} \to H$ be a multiaffine map. Let $Z = \{x_{[k]} \in G_{[k]} \colon B(x_{[k]}) = 0\}$. Then, there are non-empty, disjoint Bohr varieties $B_1, \dots, B_m \subset B$ of codimension at most $C \log_p^D \epsilon^{-1}$ such that $\Big|B \setminus (\cup_{i \in [m]} B_i)\Big| \leq \epsilon |G_{[k]}|.$\end{theorem}

%\begin{theorem}[Inner approximation of Bohr varieties - \textbf{Inner}($k$)]\label{innerAppThm}For given $k$, there are constants $C=C^{\bm{inner}}_k, D=D^{\bm{inner}}_k > 0$ with the following property. Let $\epsilon > 0$ and let $B \colon G_{[k]} \to H$ be a multiaffine map. Let $Z = \{x_{[k]} \in G_{[k]} \colon B(x_{[k]}) = 0\}$. Then there is $r \leq C \log_p^D \epsilon^{-1}$, a multiaffine map $\beta \colon G_{[k]} \to \mathbb{F}_p^r$ and some layers $L_1, \dots, L_m$ of $\beta$ such that $L_i \subset B$ for all $i \in [m]$ and $\Big|B \setminus (\cup_{i \in [m]} L_i)\Big| \leq \epsilon |G_{[k]}|.$\end{theorem}

Let $S \subset G_{[k]}$ and let $\alpha \colon G_{[k]} \to H$ be a multiaffine map. A layer of $\alpha$ is any set of the form $\{x_{[k]} \in G_{[k]} \colon \alpha(x_{[k]}) = \lambda\}$, for $\lambda \in H$. We say that layers of $\alpha$ \emph{internally $\epsilon$-approximate }$S$, if there are layers $L_1, \dots, L_m$ of $\alpha$ such that $S \supset L_i$ and $\Big|S \setminus \Big(\cup_{i \in [m]} L_i\Big)\Big| \leq \epsilon |G_{[k]}|$. Similarly, we say that layers of $\alpha$ \emph{externally $\epsilon$-approximate }$S$, if there are layers $L_1, \dots, L_m$ of $\alpha$ such that $S \subset \cup_{i \in [m]} L_i$ and $\Big|\Big(\cup_{i \in [m]} L_i\Big)  \setminus S\Big| \leq \epsilon |G_{[k]}|$.\\
\indent The next two results say that we may approximate internally and externally certain sets by low-codimensional varieties. In the first case, the sets we have in mind are dense varieties, and in the second case these are the sets of dense columns of a variety. 

\begin{theorem}[Simultaneous inner approximation of varieties - \textbf{Inner}($k$)]\label{innerAppThm}For given $k$, there are constants $C=C^{\bm{inner}}_k, D=D^{\bm{inner}}_k > 0$ with the following property. Let $\epsilon > 0$ and let $B_1, \dots, B_r \colon G_{[k]} \to H$ be multiaffine maps. For each $\lambda \in \mathbb{F}^r$, let $Z_\lambda = \{x_{[k]} \in G_{[k]} \colon \sum_{i \in [r]} \lambda_i B_i(x_{[k]}) = 0\}$. Then there is $s \leq C \Big(r \log_{|\mathbb{F}|}(|\mathbb{F}|\epsilon^{-1}) \Big)^D$, a multiaffine map $\beta \colon G_{[k]} \to \mathbb{F}^s$ such that for each $\lambda \in \mathbb{F}^r$ , layers of $\beta$ internally $\epsilon$-approximate $Z_\lambda$.\end{theorem}

\begin{theorem}[Structure of a set of dense columns of a variety - \textbf{Columns}($k$)]\label{denseColumnsThm}For given $k$, there are constants $C=C^{\bm{columns}}_k, D=D^{\bm{columns}}_k > 0$ with the following property. Let $\alpha\colon G_{[k]} \to \mathbb{F}^r$ be a multiaffine map. Let $S \subset \mathbb{F}^r$ and $\epsilon > 0$. Define the set of $\epsilon$-dense columns as
\[X = \{x_{[k-1]} \in G_{[k-1]} \colon |\{y \in G_k \colon \alpha(x_{[k-1]}, y) \in S\}| \geq \epsilon |G_k|\}.\]
Then, there is $s \leq C \Big(r \log_{|\mathbb{F}|} (|\mathbb{F}|\epsilon^{-1})\Big)^D$, a multiaffine map $\beta \colon G_{[k-1]} \to \mathbb{F}^s$ such that layers of $\beta$ $\epsilon$-internally and $\epsilon$-externally approximate $X$.\end{theorem}

Finally, we prove a strong approximation result for the convolutions of the indicator function of a low-codimensional variety. We call it an \emph{almost $L^\infty$ approximation theorem} which sounds slightly oxymoronical, but is appropriate since we actually prove that the convolution can be approximated by a finite exponential sum on very structured set (a union of layers of a low-codimensional variety) of density $1-o(1)$. For this theorem, we need one more piece of notation. For a map $f \colon G_{[k]} \to \mathbb{C}$, we define its \emph{convolution in direction} $i$, denoted by $\bigconv{i}f \colon G_{[k]} \to \mathbb{C}$ as
\[\bigconv{i}f(x_{[k]}) = \ex_{y_i \in G_i} f(x_{[i-1]}, y_i + x_i, x_{[i + 1, k]})\overline{f(x_{[i-1]}, y_i, x_{[i + 1, k]})}.\] 
We also misuse notation slightly and for the given set $Z$ we also treat $Z$ as the indicator function in the expression below.

\begin{theorem}[Almost $L^\infty$ approximation theorem for convolutions of varieties of low codimension - \textbf{Conv}($k$)]\label{convAppThm}For given $l \in [k]$, there are constants $C=C^{\bm{conv}}_{k, l}, D=D^{\bm{conv}}_{k, l} > 0$ with the following property. Let $\alpha \colon G_{[k]} \to \mathbb{F}^r$ be a multilinear map, $Z = \{x_{[k]} \in G_{[k]} \colon \alpha(x_{[k]}) = 0\}$ and let $\epsilon > 0$. Then there are $s, t \leq C \Big(r\log_{|\mathbb{F}|}(|\mathbb{F}|\epsilon^{-1}) \Big)^D$, multiaffine forms $\beta_i \colon G_{[k]} \to \mathbb{F}$ for $i \in [s]$, multiaffine map $\gamma \colon G_{[k] \setminus \{l\}} \to \mathbb{F}^t$, constants $c_1, \dots, c_m \in \mathbb{C}$, multiaffine maps $\rho_1, \dots, \rho_m \in \spn\{\beta_{[s]}\}$ and layers $L_1, \dots, L_n$ of $\gamma$ such that
\[\bigg| \bigconv{l}\bigconv{l-1} \cdots \bigconv{1} Z(x_{[k]}) - \sum_{i \in [m]} c_i \chi\Big(\rho_i(x_{[k]})\Big)\bigg| \leq \epsilon\]
for all $x_{[k]} \in G_{[k]} \setminus \Big((\cup_{i \in [n]} L_i) \times G_l\Big)$, $|\cup_{i \in [n]} L_i| \leq \epsilon |G_{[k] \setminus \{l\}}|$ and $\sum_{i \in [m]} |c_i| \leq 1$.
\end{theorem}

The proof naturally splits into five parts, each showing one of the following implications.
\[\begin{split}\bm{Weak}(k) \implies \bm{Strong}(k) \implies \bm{Inner}(k-1) &\implies \bm{Columns}(k) \\
\Big(\bm{Inner}(k-1) \land \bm{Columns}(k-1)\Big) &\implies \bm{Conv}(k) \implies \bm{Weak}(k+1).\end{split}\]
To complete this inductive scheme, we note that $\bm{Strong}(2)$ holds, and this is a simple consequence of linear algebra. Indeed, if $\alpha \colon G_1 \times G_2\to\mathbb{F}$ is a bilinear form such that $\ex_{x,y} \chi(\alpha(x,y)) \geq c$, writing $\alpha(x,y) = A(x) \cdot y$ for a linear map $A \colon G_1\to G_2$, we see that $|\{A=0\}|\geq c |G_1|$. By rank-nullity theorem, $A$ has rank $r \leq \log_{|\mathbb{F}|} c^{-1}$, thus, there are $v_1, \dots, v_r \in G_2$, and linear forms $\beta_1, \dots, \beta_r \colon G_1 \to \mathbb{F}$ such that $(\forall x \in G_1) A(x) = \sum_{i \in [r]} v_i \beta_i(x)$. Thus $\alpha(x,y) = \sum_{i \in [r]} \beta_i(x) (v_i \cdot y)$, as desired.\\

In some sense, all the results above can be seen as corollaries of Theorem~\ref{strongInvThm} (or any other theorem listed here), but this would not be an entirely correct viewpoint since the proof has the structure outlined here. Still, the deduction of the strong inverse theorem from the weak one occupies the largest part of the proof. The crucial idea in this part of the proof is the following proposition. To state it, we need to introduce a notion of connectivity for subsets of $G_{[k]}$. Namely, we consider $G_{[k]}$ as vertices of a graph $\mathcal{G}$ with edges between points that differ in exactly one coordinate. We say that a set $S \subset G_{[k]}$ is \emph{connected} if $\mathcal{G}[S]$ is connected. The \emph{diameter} of $S$ is the largest distance between two vertices in the graph $\mathcal{G}[S]$.

\begin{proposition}[One-sided regularity lemma]\label{nonzeroConnIntro}Let $\rho \colon G_{[k]} \to \mathbb{F}, \gamma_i \colon G_{I_i} \to \mathbb{F}$, $i \in [r]$ be multilinear maps. Let $F = \{i \in [r] \colon I_i = [k]\}$. Suppose that
\[\ex_{x_1, \dots, x_k} \chi\Big(\rho(x_{[k]}) - \sum\limits_{i \in F} \lambda_i \gamma_i(x_{[k]})\Big) \leq \eta = 2^{-2k} |\mathbb{F}|^{-(k+1)(3r+2)},\]
for any choice of $\lambda \in \mathbb{F}^F$. Then, the set $\{x_{[k]} \in G_{[k]} \colon (\forall i \in [r])\gamma_i(x_{I_i}) = 0, \rho(x_{[k]}) \not= 0\}$ is connected and of diameter at most $(2k+1)(2^k-1)$.
\end{proposition}

Thus, if the form $\rho$ is sufficiently quasirandom w.r.t.\ other forms $\gamma_i$, then the set $\{x_{[k]} \in G_{[k]} \colon (\forall i \in [r])\gamma_i(x_{I_i}) = 0, \rho(x_{[k]}) \not= 0\}$ is well-behaved. For our purposes, this means that we may easily remove $\rho$ from the collection of the considered maps. On the other hand, if neither form is sufficiently quasirandom, then we may replace them by forms that depend on fewer coordinates using the weak inverse theorem.\\

Another idea that plays a very important role in the proof is the dependent random choice, which takes a particularly simple form in the algebraic setting and allows us to externally approximate dense varieties by low-codimensional varieties very efficiently (Lemma~\ref{BohrApprox}).\\

\noindent\textbf{Acknowledgements.} I would like to acknowledge the support of the Ministry of Education, Science and Technological Development of the Republic of Serbia, Grant ON174026. I would also like to thank the anonymous referee for a very careful reading.

\boldSection{Preliminaries}

From now on, we adopt a non-standard convention and write $\log$, without subscripts, to be a slightly modified version of the logarithm. Namely, for positive real $x$, we write $\log x = \log_{|\mathbb{F}|} (|\mathbb{F}| x) = (\log_{|\mathbb{F}|} x ) + 1$. This has the merit of being greater or equal to 1, when $x \geq 1$, which simplifies the calculations. If we write $\log_{|\mathbb{F}|}$, we still have its usual meaning in mind.\\

As a warm-up, we show that low-codimensional varieties are necessarily dense. We use this very simple fact in the proofs that follow without explicitly referring to the next lemma.

\begin{lemma}\label{basicDensityL}Let $B$ be a variety of codimension $r$ in $G_{[k]}$. Let $x_{[k]} \in B$. Then there are at least $|\mathbb{F}|^{-kr}|G_{[k]}|$ points in $B$ at distance\footnote{In the induced graph $\mathcal{G}[B]$, where $\mathcal{G}$ has the same meaning as in the introduction.} at most $k$ from $x_{[k]}$. In particular, if $B$ is non-empty, then $|B| \geq |\mathbb{F}|^{-kr}|G_{[k]}|$.\end{lemma}

\begin{proof}For $i \in [k]$, we show that there is $Y_i \subset G_{[i]}$ of size $|Y_i| \geq |\mathbb{F}|^{-ir}|G_{[i]}|$ such that for each $y_{[i]} \in Y_i$, the point $(y_{[i]}, x_{[i+1,k]})$ belongs to $B$ and is at distance at most $i$ from $x_{[k]}$. Let $\beta \colon G_{[k]} \to \mathbb{F}^r$ be the multiaffine map defining $B$, thus $B = \{\beta = \lambda\}$, for some $\lambda \in \mathbb{F}^r$. For $i = 1$, we may take $Y_1 = \{y_1 \in G_1 \colon \beta(y_1, x_{[2, k]}) = \lambda\} \times \{x_{[2,k]}\}$. The projection of this set in $G_1$ is a non-empty (since it contains $x_1$) coset of codimension at most $r$, hence the claim follows. Suppose that the claim holds for some $i \leq k-1$. Similarly as in the previous case, for each $y_{[i]} \in Y_i$ look at $Z(y_{[i]}) = \{z_{i+1} \in G_{i+1} \colon \beta(y_{[i]}, z_{i+1}, x_{[i+2, k]}) = \lambda\}$, which is again non-empty (it contains $x_{i+1}$) coset of codimension at most $r$. Taking $Y_{i+1} = \cup_{y_{[i]} \in Y_i} y_{[i]} \times Z(y_{[i]}) \times \{x_{[i+2, k]}\}$ finishes the proof.\end{proof}

When $A\colon G_{[k]} \to H$ is a map, we write $\{A = 0\} = \{x_{[k]} \in G_{[k]} \colon A(x_{[k]}) = 0\}$, when there is no danger of confusion. Also, if $A_1, \dots, A_r \colon G_{[k]} \to H$ are maps, and $\lambda \in \mathbb{F}^r$, we write $\lambda \cdot A$ for the map $\sum_{i \in [r]} \lambda_i A_i$, when there is no danger of confusion.

\begin{lemma}[Approximating dense varieties externally]\label{BohrApprox} Let $A \colon G_{[k]} \to H$ be a multiaffine map. Then, there is a multiaffine map $\phi \colon G_{[k]} \to \mathbb{F}^s$ such that $\{A = 0\} \subset \{\phi = 0\}$ and $|\{\phi = 0\} \setminus \{A = 0\}| \leq |\mathbb{F}|^{-s}|G_{[k]}|$. If, additionally, $A$ is linear in coordinate $c$, then so is $\phi$.\end{lemma}

\begin{proof}Take $h_1, \dots, h_s$ uniformly and independently from $H$, and set $\phi(x_{[k]})_i = A(x_{[k]}) \cdot h_i$, $i \in [s]$. If $A$ is linear in coordinate $c$, then so is $\phi$, as required. Notice that $\{A = 0\} \subset \{\phi = 0\}$ holds immediately. On the other hand, if $x_{[k]} \in G_{[k]}$ satisfies $A(x_{[k]}) \not= 0$, then
\[\mathbb{P}(\phi(x_{[k]}) = 0) = \mathbb{P}\Big((\forall i \in [s]) A(x_{[k]}) \cdot h_i = 0\Big) = |\mathbb{F}|^{-s}.\]
Thus, $\ex |\{\phi = 0\} \setminus \{A = 0\}| \leq |\mathbb{F}|^{-s}|\{A \not= 0\}|$, and the claim follows.\end{proof}

\begin{lemma}[Approximating dense varieties simultaneously.]\label{BohrApproxSim}Let $A_1, \dots, A_r \colon G_{[k]} \to H$ be multiaffine maps. Let $\epsilon > 0$. Then, there is $s \leq r + \log_{|\mathbb{F}|} \epsilon^{-1}$, multiaffine maps $\phi_1, \dots, \phi_r \colon G_{[k]} \to \mathbb{F}^s$ such that for each $\lambda \in \mathbb{F}^r$ we have $\{\lambda \cdot A = 0\} \subset \{\lambda \cdot \phi = 0\}$ and $|\{\lambda \cdot \phi = 0\} \setminus \{\lambda \cdot A = 0\}| \leq \epsilon |G_{[k]}|$.\end{lemma}

\begin{proof}Consider an auxiliary multiaffine map $\Phi \colon G_{[k]} \times \mathbb{F}^r \to H$, defined by
\[\Phi(x_{[k]}, \lambda) = \sum_{i \in [r]} \lambda_i A_i(x_{[k]}).\]
Apply Lemma~\ref{BohrApprox} to find $s \leq \log_{|\mathbb{F}|}( |\mathbb{F}|^r \epsilon^{-1}) = r + \log_{|\mathbb{F}|}\epsilon^{-1}$ and a multiaffine map $\psi \colon G_{[k]} \times \mathbb{F}^r \to \mathbb{F}^s$ such that $\{\psi = 0\} \supset \{\Phi = 0\}$ and the difference set has density at most $|\mathbb{F}|^{-r} \epsilon$ in $G_{[k]} \times \mathbb{F}^r$. Furthermore, since $\Phi$ is linear in the last (auxiliary) coordinate, so is $\psi$. If $e_1, \dots, e_r$ is the standard basis of $\mathbb{F}^r$, let $\phi_i$ be defined by $\phi_i(x_{[k]}) = \psi(x_{[k]}, e_i)$. Hence, for each $\lambda \in \mathbb{F}^r$, we have that
\[\begin{split}\Big\{x_{[k]} \in G_{[k]} \colon \sum_{i \in [r]} \lambda_i \phi_i(x_{[k]}) = 0\Big\} &= \Big\{x_{[k]} \in G_{[k]} \colon \sum_{i \in [r]} \lambda_i \psi(x_{[k]}, e_i) = 0\Big\} = \{x_{[k]} \in G_{[k]} \colon \psi(x_{[k]}, \lambda) = 0\}\\
&\supset \{x_{[k]} \in G_{[k]} \colon \Phi(x_{[k]}, \lambda) = 0\}  = \Big\{x_{[k]} \in G_{[k]} \colon \sum_{i \in [r]} \lambda_i A_i(x_{[k]}) = 0\Big\}\end{split}\]
and
\[\begin{split}&\Big|\Big\{x_{[k]} \in G_{[k]} \colon \sum_{i \in [r]} \lambda_i \phi_i(x_{[k]}) = 0\Big\} \setminus \Big\{x_{[k]} \in G_{[k]} \colon \sum_{i \in [r]} \lambda_i A_i(x_{[k]}) = 0\Big\}\Big|\\
\leq &\Big|\Big\{(x_{[k]},\mu) \in G_{[k]}\times \mathbb{F}^r \colon\psi(x_{[k]}, \mu) = 0\Big\} \setminus \Big\{(x_{[k]},\mu) \in G_{[k]}\times \mathbb{F}^r \colon \Phi(x_{[k]}, \mu) = 0\Big\}\Big|\\
\leq &|\mathbb{F}|^{-r} \epsilon |G_{[k]} \times \mathbb{F}^r| = \epsilon |G_{[k]}|,\end{split}\]
as desired.\end{proof}

When $A\colon G_{[k]} \to H$ is a multiaffine map, we may write $A(x_{[k]}) = \sum_{I \subset [k]} A_I(x_I)$, for multilinear maps $A_I \colon G_I \to H$ (for $I = \emptyset$, $A_I$ is a constant, but not necessarily zero). We call $A_I$ the \emph{multilinear parts} of $A$. We make use of the following observation of Lovett~\cite{Lov}.

\begin{lemma}[Lovett~\cite{Lov}]\label{biasHomog}Let $\alpha \colon G_{[k]} \to \mathbb{F}$ be a multiaffine form, with multilinear parts $\alpha_I$. Then
\[\Big|\ex_{x_{[k]}} \chi(\alpha(x_{[k]}))\Big| \leq \ex_{x_{[k]}} \chi(\alpha_{[k]}(x_{[k]})) \in \mathbb{R}_{\geq 0}.\]\end{lemma}

\begin{proof}[Sketch proof] Write $\alpha(x_{[k]}) = \alpha'(x_{[k-1]}) + A(x_{[k-1]}) \cdot x_k$, for multiaffine maps $\alpha'$ and $A$. Then
\[\begin{split}\Big|\ex_{x_{[k]}} \chi(\alpha(x_{[k]}))\Big| = &\Big|\ex_{x_{[k-1]}} \chi(\alpha'(x_{[k-1]})) \ex_{x_k} \chi\Big(A(x_{[k-1]}) \cdot x_k\Big)\Big| = \Big|\ex_{x_{[k-1]}} \chi(\alpha'(x_{[k-1]})) \bm{1}(A(x_{[k-1]}) = 0)\Big|\\
\leq &\ex_{x_{[k-1]}} \bm{1}(A(x_{[k-1]}) = 0) = \ex_{x_{[k]}} \chi\Big(\sum_{k \in I \subset [k]} \alpha_I(x_I)\Big).\end{split}\]
Apply this observation $k-1$ more times to end up with $\alpha_{[k]}$ only in the final bound.\end{proof}

Recall the definition of Gowers box norm (Definition B.1 in the Appendix B of~\cite{GreenTaoPrimes})
\[\|f\|_{\square^{k}}^{2^k} = \ex_{x_{[k]}, y_{[k]} \in G_{[k]}} \prod_{I \subset [k]} \operatorname{Conj}^{|I|} f(x_I, y_{[k] \setminus I})\]
for a map $f \colon G_{[k]} \to \mathbb{C}$, where $\operatorname{Conj}$ stands for the conjugation operator. This definition is a generalization of Gowers uniformity norms, which were introduced by Gowers in~\cite{GowSze} in additive combinatorics and by Host and Kra in~\cite{HostKra} in the context of ergodic theory. For a more detailed discussion of box norms, see~\cite{GreenTaoPrimes}. Note that when $\phi \colon G_{[k]} \to \mathbb{F}$ is a multilinear form, we have $\|\chi \circ \phi\|_{\square^{k}}^{2^k} = \ex_{x_{[k]} \in G_{[k]}} \chi(\phi(x_{[k]}))$. The following is the Gowers-Cauchy-Schwarz inequality for the box norm.

\begin{proposition}\label{unifBound}Let $f_I \colon G_{[k]} \to \mathbb{C}$ be a function for each $I \subset [k]$. Then
\[\Big|\ex_{x_{[k]}, y_{[k]} \in G_{[k]}} \prod_{I \subset [k]} \operatorname{Conj}^{|I|} f_I(x_I, y_{[k] \setminus I})\Big| \leq \prod_{I \subset [k]} \|f_I\|_{\square^k}.\]
\end{proposition}

This can be proved by induction on $k$, Cauchy-Schwarz and H\"{o}lder inequalities, for details, see~\cite{GreenTaoPrimes}. Note that Proposition~\ref{unifBound} also implies a bound resembling that in Lemma~\ref{biasHomog}, but such a bound would be weaker than that in Lemma~\ref{biasHomog}. In fact, in the rest of the paper we only need Lemma~\ref{biasHomog}, but we include the definition of Gowers box norm and Proposition~\ref{unifBound}, since we apply it in the introduction in the sketch proof of Theorem~\ref{biasedPoly}.

\boldSection{$\bm{Weak}(k) \implies \bm{Strong}(k)$}
\noindent\emph{Proof.} We begin the deduction of the strong inverse theorem from the weak one by proving the `one-sided regularity lemma'.
\begin{proposition}\label{nonzeroConn} Let $\rho \colon G_{[k]} \to \mathbb{F}, \gamma_i \colon G_{I_i} \to \mathbb{F}$, $i \in [r]$ be multilinear maps. Let $F = \{i \in [r] \colon I_i = [k]\}$. Suppose that
\[\ex_{x_1, \dots, x_k} \chi\Big(\rho(x_{[k]}) - \sum\limits_{i \in F} \lambda_i \gamma_i(x_{[k]})\Big) \leq \eta = 2^{-2k} |\mathbb{F}|^{-(k+1)(3r+2)},\]
for any choice of $\lambda \in \mathbb{F}^F$. Then, the set $\{x_{[k]} \in G_{[k]} \colon (\forall i \in [r])\gamma_i(x_{I_i}) = 0, \rho(x_{[k]}) \not= 0\}$ is connected and of diameter at most $(2k+1)(2^k-1)$.
\end{proposition}

\begin{proof}Write $r = r_0 + r_1$ and reorder maps so that $k \in I_i$ if and only if $i \in [r_0 + 1, r]$. Also, write $S_0 = \{x_{[k-1]} \in G_{[k-1]} \colon (\forall i \in [r_0])\gamma_i(x_{I_i}) = 0\}$ and $S_1 = \{x_{[k]} \in G_{[k]} \colon (\forall i \in [r_0 + 1, r])\gamma_i(x_{I_i}) = 0, \rho(x_{[k]}) \not= 0\}$. The set we are interested then becomes $S = (S_0 \times G_k) \cap S_1.$ We first prove that for almost all of pairs $(x_{[k-1]}),$ $(y_{[k-1]}) \in G_{[k-1]}$ we have some $z \in G_k$ such that $(x_{[k-1]}, z), (y_{[k-1]}, z) \in S_1$.\\
\indent We may find multilinear maps $\Gamma_i \colon G_{I_i \setminus \{k\}} \to G_k,$ $i \in [r_0 + 1, r]$ and $R \colon G_{[k-1]} \to G_k$ such that
\[\gamma_i(x_{I_i}) = \Gamma_i(x_{I_i \setminus \{k\}}) \cdot x_k\text{ and }\rho(x_{[k]}) = R(x_{[k-1]}) \cdot x_k.\]
Observe that if $x_{[k-1]}, y_{[k-1]} \in G_{[k-1]}$ are such that
\[R(x_{[k-1]}), R(y_{[k-1]}) \notin \spn\Big\{\Gamma_i(x_{I_i \setminus \{k\}}), \Gamma_i(y_{I_i \setminus \{k\}}) \colon i \in [r_0 + 1, r]\Big\},\]
then we may certainly get a $z \in G_k$ such that $(x_{[k-1]}, z), (y_{[k-1]}, z) \in S_1$. For the sake of completeness, we include a short proof.

\begin{observation}\label{nonzeroObs}Let $G$ be a $\mathbb{F}$-vector space with a dot product $\cdot$. Let $v_1, v_2, u_1, \dots, u_m \in G$ be elements such that $v_1, v_2 \notin \spn\{u_1, \dots, u_m\}$. Then we have $z \in G$ such that $v_1 \cdot z, v_2 \cdot z \not= 0$, but $u_i \cdot z = 0$ for all $i \in [m]$.\end{observation}

\begin{proof}[Proof of Observation~\ref{nonzeroObs}.] Suppose contrary, for any $z$ with $u_i \cdot z = 0$ for all $i \in [m]$, we have $v_1 \cdot z = 0$ or $v_2 \cdot z = 0$. Suppose that we have $z_1$ and $z_2$ such that $u_i \cdot z_1 = 0$ and $u_i \cdot z_2 = 0$ for all $i \in [m]$, but $v_1 \cdot z_1 \not= 0, v_2 \cdot z_2 \not= 0$. Then $v_2 \cdot z_1 = 0, v_1 \cdot z_2 = 0$ and hence $(\forall i \in [m]) u_i \cdot (z_1 + z_2) = 0$, $v_1 \cdot (z_1 + z_2) \not =0, v_2 \cdot (z_1 + z_2) \not= 0$, which is a contradiction. Hence, w.l.o.g.\ we have that whenever $u_i \cdot z = 0$ for all $i \in [m]$, then $v_1 \cdot z = 0$. But then $v_1 \in \spn \{u_1, \dots, u_m\}$, which is the final contradiction.\end{proof}

We count the number of pairs $x_{[k-1]}, y_{[k-1]} \in G_{[k-1]}$ such that
\[R(x_{[k-1]}) \in \spn\Big\{\Gamma_i(x_{I_i \setminus \{k\}}), \Gamma_i(y_{I_i \setminus \{k\}}) \colon i \in [r_0 + 1, r]\Big\},\]
by counting for each linear combination $\lambda, \mu \in \mathbb{F}^{[r_0 + 1, r]}$ how often
\[R(x_{[k-1]}) = \sum_{i \in [r_0 + 1, r]} \lambda_i \Gamma_i(x_{I_i \setminus \{k\}}) + \sum_{i \in [r_0 + 1, r]} \mu_i \Gamma_i(y_{I_i \setminus \{k\}})\]
happens (and analogously for $R(y_{[k-1]})$). The density of such pairs is exactly
\begin{equation}\label{explDensRel}\ex_{x_{[k-1]}, y_{[k-1]}} \ex_t \chi\Big(\rho(x_{[k-1]}, t) - \sum\limits_{i \in [r_0 + 1, r]} \lambda_i \gamma_i(x_{I_i \setminus \{k\}}, t) - \sum\limits_{i \in [r_0 + 1, r]} \mu_i \gamma_i(y_{I_i \setminus \{k\}}, t)\Big).\end{equation}

Using Lemma~\ref{biasHomog}, we may bound~\eqref{explDensRel} by
\[\begin{split}&\ex_{y_{[k-1]}} \Big|\ex_{x_{[k-1]}, t}\chi\Big(\rho(x_{[k-1]}, t) - \sum\limits_{i \in [r_0 + 1, r]} \lambda_i \gamma_i(x_{I_i \setminus \{k\}}, t) - \sum\limits_{i \in [r_0 + 1, r]} \mu_i \gamma_i(y_{I_i \setminus \{k\}}, t)\Big)\Big| \\
&\hspace{1cm}\leq \ex_{y_{[k-1]}} \Big|\ex_{x_{[k-1]}, t}\chi\Big(\rho(x_{[k-1]}, t) - \sum\limits_{i \in F} \lambda_i\gamma_i(x_{[k-1]}, t)\Big| \leq \ex_{y_{[k-1]}} \eta  = \eta,\end{split}\]
where $F \subset [r]$ is the set of all $i$ such that $I_i = [k]$. From this we deduce that for all but at most $2|\mathbb{F}|^{2r}\eta$ proportion of pairs $x_{[k-1]},$ $y_{[k-1]} \in G_{[k-1]}$ we have some $z \in G_k$ such that $(x_{[k-1]}, z),$ $(y_{[k-1]}, z) \in S_1$. Moreover, for any given pair, we have at least $|\mathbb{F}|^{-2(r+1)}|G_k|$ such $z$, since for $\lambda_1, \lambda_2 \in \mathbb{F}$
\[\Big\{z \in G_k \colon \Big((\forall i \in [r_0 + 1, r]) \gamma_i(x_{I_i \setminus \{k\}}, z) = \gamma_i(y_{I_i \setminus \{k\}}, z) = 0\Big), \rho(x_{[k-1]}, z) = \lambda_1, \rho(y_{[k-1]}, z) = \lambda_2\Big\}\]
is an at most $2(r+1)$-codimensional coset.\\

Write $F' = \{i \in [r] \colon I_i = [k-1]\}$. Fix $t \in G_k$, and consider analytic rank of the map $\tau_{t, \lambda} \colon G_{[k-1]} \to \mathbb{F}$ defined by
\[\tau_{t,\lambda}(x_{[k-1]}) = \rho(x_{[k-1]}, t) - \sum_{i \in F'}\lambda_i \gamma_i(x_{[k-1]}) - \sum_{i \in F} \lambda_i \gamma_i(x_{[k-1]}, t)\]
for $\lambda \in \mathbb{F}^{F \cup F'}$. If analytic ranks of $\tau_{t,\lambda}$ are small for all choices of $\lambda \in \mathbb{F}^{F \cup F'}$, then the induction hypothesis applies, and $S_{k \colon t}$ (recall that this is the slice notation $S_{k \colon t} = \{z_{[k-1]} \in G_{[k-1]} \colon (z_{[k-1]}, t) \in S\}$) is connected and of diameter at most $(2k-1)(2^{k-1}-1)$. For any fixed $\lambda \in \mathbb{F}^{F \cup F'}$, we have
\[\begin{split} &\ex_{t} \bigg|\ex_{x_1, \dots, x_{k-1}}\chi\Big(\rho(x_{[k-1]}, t) - \sum_{i \in F'}\lambda_i \gamma_i(x_{[k-1]}) - \sum_{i \in F} \lambda_i \gamma_i(x_{[k-1]}, t)\Big)\bigg|\end{split}\]
\[\begin{split} =&\ex_{t} \ex_{x_1, \dots, x_{k-1}}\chi\Big(\rho(x_{[k-1]}, t) - \sum_{i \in F'}\lambda_i \gamma_i(x_{[k-1]}) - \sum_{i \in F} \lambda_i \gamma_i(x_{[k-1]}, t)\bigg)\\
&\hspace{2cm}\text{by Lemma~\ref{biasHomog}}\\
\leq &\ex_{x_1, \dots, x_{k-1}, t} \chi\Big(\rho(x_{[k-1]}, t) - \sum_{i \in F} \lambda_i \gamma_i(x_{[k-1]}, t)\Big)\\
\leq & \eta\end{split}\]
hence, by averaging, we obtain a set $T \subset G_k$ of density at least $1 - \frac{1}{2} |\mathbb{F}|^{-2r -2}$, such that
\begin{equation}\label{xkrank}(\forall t \in T)(\forall \lambda \in \mathbb{F}^{F \cup F'})\ex_{x_1, \dots, x_{k-1}} \chi\Big(\rho(x_{[k-1]}, t) - \sum\limits_{i \in F'} \lambda_i \gamma_i(x_{[k-1]}) - \sum\limits_{i \in F} \lambda_i \gamma_i(x_{[k-1]}, t)\Big) \leq 2 |\mathbb{F}|^{3r + 2}\eta.\end{equation}
Consequently, by induction hypothesis, for each $z \in T$, $S_{k \colon z}$ is connected and of diameter at most $(2k-1)(2^{k-1}-1)$.\\

We are now ready to prove that $S$ is connected and of bounded diameter. We do this in two steps, the first one being to show that this holds for a very large subset of $S$, and then the second is to extend this to whole $S$.\\

Let $X \subset G_{[k-1]}$ be the set of all $x_{[k-1]} \in G_{[k-1]}$ such that for proportion at least $1 - 2 |\mathbb{F}|^{r}\sqrt{\eta}$ of $y_{[k-1]} \in G_{[k-1]}$, we have at least $|\mathbb{F}|^{-2(r+1)}|G_k|$ of $z \in G_k$ such that $(x_{[k-1]},z), (y_{[k-1]}, z) \in S_1$. By the previous argument, we get $|X| \geq \Big(1 - 2|\mathbb{F}|^{r}\sqrt{\eta}\Big)|G_{[k-1]}|$. We claim that $(X \times G_k) \cap S$ is connected and of diameter at most $(2k-1)(2^k-2) + 3$. Indeed, let $x_{[k]}, y_{[k]} \in (X \times G_k) \cap S$. Since $|S_0| \geq |\mathbb{F}|^{-(k-1)r}|G_{[k-1]}| > 4|\mathbb{F}|^{r}\sqrt{\eta} |G_{[k-1]}|$, by the way we defined $X$, we may find some $u_{[k-1]} \in S_0$ such that we have at least $|\mathbb{F}|^{-2(r+1)}|G_k|$ of $z \in G_k$ such that $(x_{[k-1]},z),$ $(u_{[k-1]}, z) \in S_1$, and we also have at least $|\mathbb{F}|^{-2(r+1)}|G_k|$ of $z' \in G_k$ such that $(y_{[k-1]},z'),$ $(u_{[k-1]}, z') \in S_1$. In particular, recalling that $|T| \geq \Big(1 - \frac{1}{2}|\mathbb{F}|^{-2r -2}\Big)|G_{[k]}|$, we have a choice of $z, z' \in T$ with the above properties. But, $S_{k \colon z}$ and $S_{k \colon z'}$ are connected and of diameter at most $(2k-1)(2^{k-1}-1)$, which completes the first step.\\ 

Finally, take any $x_{[k]} \in S$. Since $x_{[k]}$ is at distance at most $k$ to at least $|\mathbb{F}|^{-k(r+1)}|G_{[k]}|$ of points in $S$, and $|S \setminus (X \times G_k)| \leq |G_{[k-1]} \setminus X| |G_k| \leq 2|\mathbb{F}|^{r}\sqrt{\eta} |G_{[k]}|$, at least one such point lies in $(X \times G_k) \cap S$, and we are done, with the final diameter bound being $(2k-1)(2^k-2) + 3 + 2k\leq (2k+1)(2^k-1)$.\end{proof}

Let $\overline{C} = C^{\bm{weak}}_k$ and $\overline{D} = D^{\bm{weak}}_k$. 

\begin{proof}[Proof of Theorem~\ref{strongInvThm}]Apply Theorem~\ref{weakInvThm} to $\alpha$ to find $m_0 \leq \overline{C} \log^{\overline{D}} c^{-1}$ and multilinear $\beta_i \colon G_{I_i} \to \mathbb{F}$, $i \in [m_0]$ where $\emptyset \not= I_i \subset [k-1]$ and
\[\Big\{x_{[k]} \in G_{[k]} \colon (\forall i \in [m_0]) \beta_i(x_{I_i}) = 0\Big\} \subset \{\alpha = 0\}.\]
Write $\alpha(x_{[k]}) = A(x_{[k-1]}) \cdot x_k$ for a multilinear map $A \colon G_{[k-1]} \to G_k$. Thus, we also have
\[\Big\{x_{[k-1]} \in G_{[k-1]} \colon (\forall i \in [m_0]) \beta_i(x_{I_i}) = 0\Big\} \subset \{A = 0\}.\]
For a set $Q$, the \emph{power-set} of $Q$ is the collection of all subsets of $Q$ (including $\emptyset$ and $Q$ itself) and is denoted by $\mathcal{P}Q$. By induction on the size of up-set\footnote{Collection of sets closed under taking supersets.} $\mathcal{F} \subset \mathcal{P}[k-1]$, we prove the following proposition.

\begin{proposition}\label{strongInverseInductionStep}Let $\mathcal{F} \subset \mathcal{P}[k-1]$ be an up-set. Then there are constants $C_{\mathcal{F}}, D_{\mathcal{F}}$ with the following property. We may find $m_{\mathcal{F}}, n_{\mathcal{F}} \leq C_{\mathcal{F}} \log^{D_{\mathcal{F}}} c^{-1}$, a collection of multilinear maps $\rho^{\mathcal{F}}_i \colon G_{J^{\mathcal{F}}_i} \to \mathbb{F}$, where $i \in [n_{\mathcal{F}}]$, $\emptyset\not= J^{\mathcal{F}}_i \subset [k-1]$, points $y^{\mathcal{F}, i}_{J^{\mathcal{F}}_i} \in G_{J^{\mathcal{F}}_i}$, another collection of multilinear maps $\beta^{\mathcal{F}}_i \colon G_{I^{\mathcal{F}}_i} \to \mathbb{F}$, where $i \in [m_{\mathcal{F}}]$, $\emptyset\not= I^{\mathcal{F}}_i \in \mathcal{P}[k-1] \setminus \mathcal{F}$, such that the multilinear map $A^{\mathcal{F}} \colon G_{[k-1]} \to G_k$ defined as
\[A^{\mathcal{F}}(x_{[k-1]}) = A(x_{[k-1]}) - \sum_{i \in [n_{\mathcal{F}}]} \rho^{\mathcal{F}}_i(x_{J^{\mathcal{F}}_i})A(y^{\mathcal{F}, i}_{J^{\mathcal{F}}_i}, x_{[k-1] \setminus J^{\mathcal{F}}_i})\]
satisfies
\[\Big\{x_{[k-1]} \in G_{[k-1]} \colon (\forall i \in [m_{\mathcal{F}}]) \beta^{\mathcal{F}}_i(x_{I^{\mathcal{F}}_i}) = 0\Big\} \subset \Big\{x_{[k-1]} \in G_{[k-1]} \colon A^{\mathcal{F}}(x_{[k-1]})= 0\Big\}.\]
We may take
\[C_{\mathcal{F}} = \Big(3 \overline{C}(7k)^{\overline{D}}\Big)^{(\overline{D} + 1)^{|\mathcal{F}|}}\hspace{1cm}\text{and}\hspace{1cm}D_{\mathcal{F}} = \overline{D} (\overline{D} + 1)^{|\mathcal{F}|}.\]
%\[m_{\mathcal{F} \cup \{S\}} \leq m_{\mathcal{F}} \textbf{Bnd}_{\text{weak}}\Big(|S|, p^{-2^{|S|+1}m_{\mathcal{F}}^2-2^{|S|+1}m_{\mathcal{F}}} \eta^{2^{|S|}}\Big) + m_{\mathcal{F}}\]
\end{proposition}

\begin{proof}[Proof of Proposition~\ref{strongInverseInductionStep}]For $\mathcal{F} = \emptyset$, we take the given maps $\beta_i$, and hence $n_\emptyset = 0, m_\emptyset = m_0$. Let $\mathcal{F} \cup \{S\}$ be a given up-set, where $S$ is a minimal set inside it, thus making $\mathcal{F}$ an up-set as well. Assume that the claim holds for $\mathcal{F}$, and that we get multilinear maps $\beta^{\mathcal{F}}_i \colon G_{I^{\mathcal{F}}_i} \to \mathbb{F}$, $i \in [m_{\mathcal{F}}]$, and $A' = A^{\mathcal{F}}$, with the property above. If no $i \in [m_{\mathcal{F}}]$ satisfies $I^{\mathcal{F}}_i = S$, then the same collection works for $\mathcal{F} \cup \{S\}$. Thus, after reordering maps $\beta^{\mathcal{F}}_i$ if necessary, assume that $I^{\mathcal{F}}_i = S$ if and only if $i \in [s]$. Let $\{\lambda^1, \dots, \lambda^d\} \subset \mathbb{F}^s$ be a maximal independent set such that for each $j \in [d]$
\[\ex_{x_S} \chi\Big(\sum_{i \in [s]}\lambda^j_i \beta^{\mathcal{F}}_i(x_S)\Big) \geq \eta = 2^{-2|S|} |\mathbb{F}|^{-(|S|+1)(3m_{\mathcal{F}}+2)}.\]
Thus, if we extend $\lambda_1, \dots, \lambda_d$ by further $\mu^1, \dots, \mu^{s-d}$ to a basis of $\mathbb{F}^s$, and setting 
\[\gamma_i(x_S) = \sum_{j \in [s]} \lambda^i_j \beta^{\mathcal{F}}_j(x_{S})\]
for $i \in [d]$, and 
\[\rho_i(x_S) = \sum_{j \in [s]} \mu^i_j \beta^{\mathcal{F}}_j(x_S)\]
for $i \in [s-d]$, then we have the following properties:
\begin{itemize}
\item[\textbf{(i)}]$(\forall i \in [d])$, $\ex_{x_S} \chi(\gamma_i(x_S)) \geq \eta$,
\item[\textbf{(ii)}]$(\forall \nu \in \mathbb{F}^d$, $\tau \in \mathbb{F}^{s-d} \setminus \{0\})$, $\ex_{x_S} \chi\Big(\sum_{i \in [d]} \nu_i \gamma_i(x_S) + \sum_{i \in [s-d]} \tau_i \rho_i(x_S)\Big) \leq \eta$,
\item[\textbf{(iii)}]$\bigg\{x_{[k-1]} \in G_{[k-1]} \colon (\forall i \in [d]) \gamma_i(x_S) = 0, (\forall i \in [s-d]) \rho_i(x_S) = 0, (\forall i \in [s + 1, m_{\mathcal{F}}]) \beta^{\mathcal{F}}_i(x_{I^{\mathcal{F}}_i}) = 0\bigg\}$\\
$\phantom{a}\hspace{1cm}\subset \{x_{[k-1]} \in G_{[k-1]} \colon A'(x_{[k-1]}) = 0\}.$
\end{itemize}
We first deal with $\rho_i$. Let $F = \{i \in [s+1, m_{\mathcal{F}}] \colon I^{\mathcal{F}}_i \subset [k-1] \setminus S\}$ and $Z = \Big\{x_{[k-1] \setminus S} \in G_{[k-1] \setminus S} \colon (\forall i \in F) \beta_i(x_{I^{\mathcal{F}}_i}) = 0\Big\}$. The property \textbf{(ii)} and Proposition~\ref{nonzeroConn} imply that for each $z_{[k-1] \setminus S} \in G_{[k-1] \setminus S}$, for each $i \in [s-d]$, the set 
\[\begin{split}\Big\{x_{S} \in G_S \colon\hspace{3pt}\rho_i(x_S) \not= 0,\hspace{3pt}(\forall j \in [i + 1, s-d]) \rho_j(x_S) = 0,\hspace{3pt}&(\forall j \in [d]) \gamma_j(x_S) = 0,\\
 &(\forall j \in [s+1, m_{\mathcal{F}}] \setminus F) \beta^{\mathcal{F}}_j(x_{S \cap I^{\mathcal{F}}_j}, z_{I^{\mathcal{F}}_j \setminus S}) = 0\Big\}\end{split}\]
is connected.

%\[\Big(\forall \nu \in \mathbb{F}_p^d, \tau \in \mathbb{F}_p^{s-d} \setminus \{0\}, \sigma \in \mathbb{F}_p^{[s+1, m_{\mathcal{F}}]}\Big) \ex_{x_{[k]}} \omega^{\sum_{i \in [d]} \nu_i \gamma_i(x_S) + \sum_{i \in [s-d]} \tau_i \rho_i(x_S) + \sum_{i \in [s+1, m_{\mathcal{F}}]} \sigma_i \beta^{\mathcal{F}}_i(x_{J^{\mathcal{F}}_i})} \leq \eta^{2^{-|S|}},\]
%by Proposition~\ref{unifBound}.\\
%
%We first deal with $\rho_i$. By Proposition~\ref{nonzeroConn}, there is a set $Z \subset G_{[k-1] \setminus S}$ of density at least $1 - \epsilon p^{-m^2_{\mathcal{F}} - m_{\mathcal{F}}}$ such that for each $z_{[k-1]\setminus S} \in Z$, the set
%\[\begin{split}\Big\{x_{S} \in G_S \colon\hspace{3pt}\rho_i(x_S) \not= 0,\hspace{3pt}(\forall j \in [i + 1, s-d]) \rho_j(x_S) = 0,\hspace{3pt}&(\forall j \in [d]) \gamma_j(x_S) = 0,\\
 %&(\forall j \in [s+1, m_{\mathcal{F}}]) \beta^{\mathcal{F}}_j(x_{S \cap I^{\mathcal{F}}_j}, z_{I^{\mathcal{F}}_j \setminus S}) = 0\Big\}\end{split}\]
%is connected for each $i \in [s-d]$.

Next, we pick points $y^i_S \in G_S$ for each $i \in [s-d]$ and define
\[V = \Big\{z_{[k-1] \setminus S} \in G_{[k-1] \setminus S} \colon (\forall i \in [s-d])(\forall j \in [s+1, m_{\mathcal{F}}] \setminus F) \beta_j^{\mathcal{F}}(y^i_{S \cap I^{\mathcal{F}}_j}, z_{I^{\mathcal{F}}_j \setminus S}) = 0\Big\}.\]
We claim that we may choose the points $y_S^1, \dots, y_S^{s-d} \in G_S$ so that
\[(\forall i \in [s-d]) \rho_i(y^i_S) = 1, (\forall j \in [s-d] \setminus \{i\}) \rho_j(y^i_S) = 0, (\forall l \in [d]) \gamma_l(y^i_S) = 0.\]
%\begin{itemize}
%\item[\textbf{(i)}] $(\forall i \in [s-d]) \rho_i(y^i_S) = 1$, $(\forall j \in [s-d] \setminus \{i\}) \rho_j(y^i_S) = 0$, $(\forall j \in [d]) \gamma_j(y^i_S) = 0$, and
%\item[\textbf{(ii)}]$|V \cap Z| \geq (1-\epsilon)|V|$.
%\end{itemize}
Indeed, we may certainly satisfy this condition since otherwise get that some $\rho_i = 0$ whenever the maps $\rho_j$ for $j \not= i$ and $\gamma_j$ are all zero, and we may discard $\rho_i$.\\
%The second property is automatically satisfied, since $|V| \geq p^{-sm_{\mathcal{F}}} |G_{[k-1]\setminus S}| \geq p^{-m^2_{\mathcal{F}}} |G_{[k-1] \setminus S}|$ and $|Z^c| \leq \epsilon p^{-m^2_{\mathcal{F}}} |G_{[k-1] \setminus S}|$. \\

Next, we set
\[W = \Big\{x_{[k-1]} \in G_{[k-1]} \colon (\forall j \in [d]) \gamma_j(x_S) = 0, (\forall j \in [s+1, m_{\mathcal{F}}] \setminus F) \beta^{\mathcal{F}}_j(x_{I^{\mathcal{F}}_j}) = 0\Big\}\]
and we observe the following.
\begin{proposition}\label{layerEquation}For all $x_{[k-1]} \in W \cap (G_S \times (Z \cap V))$ we have the equality
\begin{equation}\label{leq}A'(x_{[k-1]}) = \sum_{i \in [s-d]} \rho_i(x_S) A'(y^i_S, x_{[k-1] \setminus S}).\end{equation}
\end{proposition}

\begin{proof}[Proof of Proposition~\ref{layerEquation}]Fix $z_{[k-1] \setminus S} \in Z \cap V$. We show that for all $x_S \in W_{z_{[k-1] \setminus S}}$
\[A'(x_S, z_{[k-1] \setminus S}) = \sum_{i \in [s-d]} \rho_i(x_S) A'(y^i_S, z_{[k-1] \setminus S})\]
holds. We argue by induction on the maximal index $i \in [s-d]$ such that $\rho_i(x_S) \not= 0$, and if there are no such $i$, we put $i = 0$. Thus the base case is when all $\rho_i(x_S) = 0$. However, since $(x_S, z_{[k-1] \setminus S}) \in W$ and $z_{[k-1] \setminus S} \in Z$, by property \textbf{(iii)}, it follows that $A'(x_S, z_{[k-1] \setminus S}) = 0$. The right hand side is also zero, so the identity holds.\\
\indent Assume now that the claim holds for values of $i$ smaller than some $i_0 \geq 1$, and that we are given $x_S \in W_{z_{[k-1] \setminus S}}$ such that $\rho_{i_0}(x_S) \not= 0$, but $\rho_{i}(x_S) = 0$ for $i > i_0$. Recall that the set
\[\begin{split}R = \Big\{g_{S} \in G_S \colon \rho_{i_0}(g_S) \not= 0,\hspace{3pt}&(\forall j \in [i_0 + 1, s-d])\hspace{2pt} \rho_j(g_S) = 0,\hspace{3pt}(\forall j \in [d])\hspace{2pt}\gamma_j(g_S) = 0,\\ 
&(\forall j \in [s+1, m_{\mathcal{F}}] \setminus F)\hspace{2pt} \beta^{\mathcal{F}}_j(g_{S \cap I^{\mathcal{F}}_j}, z_{I^{\mathcal{F}}_j \setminus S}) = 0\Big\} \subset W_{z_{[k-1] \setminus S}}\end{split}\]
is connected. Thus, since $x_S, y^{i_0}_S \in R$, there is a sequence of points $x_S = x^0_S, x^1_S, \dots, x^l_S = y^{i_0}_S$ inside this set, such that every two consecutive points differ in exactly one coordinate (we call such points \emph{neighbouring}). We introduce the following piece of notation. When $a$ and $b$ are two points in $G_{[k]}$ differing only in coordinate $i$, we write $a-b$ to be the point with coordinates $(a-b)_j = a_j = b_j$ for $j \not= i$ and $(a-b)_i = a_i - b_i$. This notation makes calculations much neater, since for a multilinear map $\phi$ on $G_{[k]}$, we have $\phi(a - b) = \phi(a) - \phi(b)$.\\
\indent Let $\tau_i = \rho_{i_0}(x^i_S)$ for $i \in [0, l]$. If $x^0_S$ and $x^1_S$ differ in coordinate $c_0$, multiply each of points $x^1_S, \dots, x^l_S$ at coordinate $c_0$ by $\tau_0 \tau_1^{-1}$. The new points still have the property that the consecutive points are either neighbouring or identical, and that they all lie in the set $R$. Misusing the notation, we keep writing $x^i_S$ for the modified points. If $c_1$ is the coordinate where $x^1_S$ and $x^2_S$ differ, multiply all points among $x^2_S, \dots, x^l_S$ by $\tau_1 \tau_2^{-1}$ at coordinate $c_1$, and proceed. Hence, we end up with a sequence $x_S = x^0_S, x^1_S, \dots, x^l_S$ in $R$ such that for each $s \in S$, $x^l_s = \sigma_s y^{i_0}_s$ for some $\sigma_s \in \mathbb{F} \setminus \{0\}$, the consecutive points are either neighbouring or identical and $(\forall i \in [0,l])\rho_{i_0}(x^i_S) = \tau_0$. Hence $\rho_j(x^i_S - x^{i+1}_S) = 0$ for $j \in [i_0, s-d]$. Thus, we get
\[\begin{split}A'(x_S, z_{[k-1] \setminus S}) &= A'(x^0_{S} - x^1_S,  z_{[k-1] \setminus S}) + \dots + A'(x^{l-1}_{S} - x^l_S, z_{[k-1] \setminus S}) + A'(x^l_S, z_{[k-1] \setminus S})\\
&= A'(x^0_{S} - x^1_S,  z_{[k-1] \setminus S}) + \dots + A'(x^{l-1}_{S} - x^l_S, z_{[k-1] \setminus S}) + \Big(\prod_{s \in S}  \sigma_s\Big) A'(y^{i_0}_S, z_{[k-1] \setminus S})\\
&\hspace{2cm}\text{applying induction hypothesis to every }x^i_{S} - x^{i+1}_S\\
&= \sum_{i \in [0,l-1]}\sum_{j \in [s-d]} \rho_j(x^i_{S} - x^{i+1}_S) A'(y^j_S, z_{[k-1] \setminus S}) + \Big(\prod_{s \in S}  \sigma_s\Big) A'(y^{i_0}_S, z_{[k-1] \setminus S})\\
&\hspace{2cm}\text{using the fact that, for all } i \not= i_0, \rho_i(y^{i_0}_S) = 0,\\
&= \sum_{i \in [0,l-1]}\sum_{j \in [s-d]} \rho_j(x^i_{S} - x^{i+1}_S) A'(y^j_S, z_{[k-1] \setminus S}) + \Big(\prod_{s \in S}  \sigma_s\Big) \sum_{i \in [s-d]} \rho_i(y^{i_0}_S)A'(y^i_S, z_{[k-1] \setminus S})\\
&= \sum_{i \in [0,l-1]}\sum_{j \in [s-d]} \rho_j(x^i_{S} - x^{i+1}_S) A'(y^j_S, z_{[k-1] \setminus S}) + \sum_{i \in [s-d]} \rho_i(x^{l}_S)A'(y^i_S, z_{[k-1] \setminus S})\\
&= \sum_{j \in [s-d]} \Big(\sum_{i \in [0,l-1]}\rho_j(x^i_{S} - x^{i+1}_S) A'(y^j_S, z_{[k-1] \setminus S}) + \rho_j(x^{l}_S)A'(y^j_S, z_{[k-1] \setminus S})\Big)\\
&= \sum_{j \in [s-d]} \rho_j(x_S)A'(y^j_S, z_{[k-1] \setminus S}),\end{split}\]
as desired.\end{proof}

%We also need the following lemma.
%
%\begin{lemma}\label{fullSubBohr}Let $U$ be a Bohr variety in $G_{[k]}$ of codimension $t$ and let $U' \subset U$ be another Bohr variety, such that $|U'| \geq \Big(1 - 3^{-k}p^{-\frac{t(k+1)^2}{2}}\Big)|U|$. Then $U = U'$.\end{lemma}
%
%\begin{proof}By induction on $k$, for $k = 1$, this is clear. Suppose that $U$ and $U'$ are such Bohr varieties in $G_{[k]}$. Let $X = \Big\{x_k \in G_k \colon |U'_{\cdots x_k}| \geq \Big(1 - 3^{-(k-1)}p^{-\frac{t k^2}{2}}\Big) |U_{\cdots x_k}|\Big\}$. Thus $|X| \geq \Big(1- \frac{p^{-t}}{3}\Big)|G_k|$. By induction hypothesis $U'_{x_k} = U_{x_k}$ when $x_k \in X$. Finally, for each point $x \in U$, for at least $\frac{2}{3}$ of $y_k \in U_{x_{[k-1]}}$, we also have that $y_k \in X$ and hence $(x_1, \dots, x_{k-1}, y_k) \in U'$. Hence $x \in U'$.\end{proof}

%Notice that the set of all $x_{[k-1]} \in W \cap (G_S \times V)$ for which equation~\eqref{leq} is a Bohr variety $B \subset W \cap (G_S \times V)$, and $W \cap (G_S \times V)$ is a Bohr variety of codimension at most $m_{\mathcal{F}}^2+m_{\mathcal{F}}$. Taking $\epsilon = p^{}$, it follows that $|B| \geq \Big(1 - 3^{-k} p^{-k^2n^2_{\mathcal{F}}}\Big) |W \cap (G_S \times V)|$, and using Lemma~\ref{fullSubBohr} we in fact get that $B = W \cap (G_S \times V)$.
Hence, the multilinear map $A^{\mathcal{F} \cup \{S\}}$ defined by
\[A^{\mathcal{F} \cup \{S\}}(x_{[k-1]}) = A'(x_{[k-1]}) - \sum_{i \in [s-d]} \rho_i(x_S) A'(y^i_S, x_{[k-1] \setminus S})\]
satisfies
\[W \cap (G_S \times (V \cap Z)) \subset \Big\{x_{[k-1]} \in G_{[k-1]} \colon A^{\mathcal{F} \cup \{S\}}(x_{[k-1]}) = 0\Big\}.\]
On the other hand, recalling the property \textbf{(i)}, applying Theorem~\ref{weakInvThm} to each $\gamma_i$, allows us to find
\[m_i \leq \overline{C}\log^{\overline{D}} {\eta}^{-1} \leq \overline{C} \Big(k (3m_{\mathcal{F}} + 4)\Big)^{\overline{D}}\]
and further multilinear maps $\beta'_{i, j} \colon G_{I_{i, j}} \to \mathbb{F}$, $j \in [m_i]$, $\emptyset \not= I_{i,j} \subset S \setminus \{\max S\}$ such that $\{x_S \in G_S\colon (\forall j \in [m_i]) \beta'_{i,j}(x_{I_{i,j}}) = 0\} \subset \{x_S \in G_S \colon \gamma_i(x_S) = 0\}$. Thus 
\[\begin{split}&\Big\{x_{[k-1]} \in G_{[k-1]} \colon \Big(\forall i \in [s-d]\hspace{2pt}\forall j \in [s+1, m_{\mathcal{F}}] \setminus F\Big)\hspace{3pt}\beta_j^{\mathcal{F}}(y^i_{S \cap I^{\mathcal{F}}_j}, x_{I^{\mathcal{F}}_j \setminus S}) = 0\Big\}\\
&\hspace{2cm} \cap \Big\{x_{[k-1]} \in G_{[k-1]} \colon \Big(\forall i \in [s+1, m_{\mathcal{F}}]\Big) \beta^{\mathcal{F}}_i(x_{I^{\mathcal{F}}_i}) = 0\Big\}\\
&\hspace{2cm} \cap \Big\{x_{[k-1]} \in G_{[k-1]} \colon \Big(\forall i \in [d]\hspace{2pt}\forall j \in [m_i]\Big) \beta'_{i, j}(x_{I_{i,j}}) = 0\Big\}\\
&\hspace{1cm}\subset \Big\{x_{[k-1]} \in G_{[k-1]} \colon A^{\mathcal{F} \cup \{S\}}(x_{[k-1]}) = 0\Big\},\end{split}\]
as desired. When it comes to bounds, we may take
\[m_{\mathcal{F} \cup \{S\}} \leq m_{\mathcal{F}}^2 + m_{\mathcal{F}} + \overline{C} m_{\mathcal{F}} \Big(k (3m_{\mathcal{F}} + 4)\Big)^{\overline{D}},\]
and
\[n_{\mathcal{F} \cup \{S\}} \leq n_{\mathcal{F}} + m_{\mathcal{F}}.\]
This finishes the proof.\end{proof}

Applying Proposition~\ref{strongInverseInductionStep} with $\mathcal{F} = \mathcal{P}[k-1]$ implies that 
\[A(x_{[k-1]}) = \sum_{i \in [n_{\mathcal{F}}]} \rho^{\mathcal{F}}_i(x_{J^{\mathcal{F}}_i})A(y^{\mathcal{F}, i}_{J^{\mathcal{F}}_i}, x_{[k-1] \setminus J^{\mathcal{F}}_i})\]
for all $x_{[k-1]} \in G_{[k-1]}$. Taking dot product with $x_k$ completes the proof.\end{proof}

Hence, we may take $C^{\bm{strong}}_k = \Big(3 \overline{C}(7 k)^{\overline{D}}\Big)^{(\overline{D} + 1)^{2^k}}$, and $D^{\bm{strong}}_k = \overline{D} (\overline{D} + 1)^{2^k}$, where $\overline{C} = C^{\bm{weak}}_k$ and $\overline{D} = D^{\bm{weak}}_k$. Hence, if $t$ is a quantity such that $C^{\bm{weak}}_k \leq 2^{k^{2^{t}}}$ and $D^{\bm{weak}}_k \leq 2^{2^{t}}$, then
\begin{equation}\label{strongBounds}C^{\bm{strong}}_k \leq 2^{k^{2^{t + O(k)}}}\hspace{1cm}\text{and}\hspace{1cm}D^{\bm{strong}}_k \leq 2^{2^{t + O(k)}}.\end{equation}
\qed

\boldSection{$\bm{Strong}(k) \implies \bm{Inner}(k-1)$}

\noindent\emph{Proof.} Let $\overline{C} = C^{\bm{strong}}_k$ and $\overline{D} = D^{\bm{strong}}_k$. The theorem will follow from the following proposition. For a given $\mathcal{F} \subset \mathcal{P}([k-1])$, we say that a multiaffine map $B \colon G_{[k-1]} \to H$ is \emph{$\mathcal{F}$-supported} if it can be written as $B(x_{[k-1]}) = \sum_{I \in \mathcal{F}} B_I(x_I)$, for some multilinear maps $B_I \colon G_I \to H$.

\begin{proposition} Let $\mathcal{F} \subset \mathcal{P}([k-1])$ be a non-empty down-set.\footnote{A collection of set closed under taking subsets.} Then, there are constants $C_{\mathcal{F}}, D_{\mathcal{F}}$ with the following property. Let $\epsilon > 0$ and let $B_1, \dots, B_r \colon G_{[k-1]} \to H$ be multiaffine maps. For $\lambda \in \mathbb{F}^r$, let $Z_\lambda = \{x_{[k]} \in G_{[k-1]} \colon \sum_{i \in [r]} \lambda_i B_i(x_{[k-1]}) = 0\}$. Then, there are $s, t \leq C_{\mathcal{F}} \Big(r \log \epsilon^{-1}\Big)^{D_{\mathcal{F}}}$, a multiaffine map $\beta \colon G_{[k-1]} \to \mathbb{F}^s$, and a collection of $\mathcal{F}$-supported multiaffine maps $\Gamma_1, \dots, \Gamma_t \colon G_{[k-1]} \to H$ such that:
\begin{itemize}
\item[\textbf{(i)}] for each $\lambda \in \mathbb{F}^r$, there are distinct layers $L_1, \dots, L_m$ of $\beta$ and multiaffine maps $A_1, \dots, A_m \in \spn \{\Gamma_{[t]}\}$, such that $Z_\lambda \cap L_i = \{x_{[k-1]} \in G_{[k-1]} \colon A_i(x_{[k-1]}) = 0\} \cap L_i$,
\item[\textbf{(ii)}] $\Big|Z_\lambda \setminus \Big(\cup_{i \in [m]} (Z_\lambda \cap L_i)\Big)\Big| \leq \frac{2^k - |\mathcal{F}|}{2^k} \epsilon |G_{[k-1]}|$.
\end{itemize}
We may take $C_{\mathcal{F}} = \Big(2 \overline{C} (2k)^{\overline{D}}\Big)^{(\overline{D} + 1)^{2^k - |\mathcal{F}|}}$ and $D_{\mathcal{F}} = (2\overline{D} + 2)^{2^k - |\mathcal{F}|}$.
\end{proposition}

\begin{proof} We prove the claim by down-wards induction on $|\mathcal{F}|$. The base case is $\mathcal{F} = \mathcal{P}([k-1])$, in which case we take $s = 1$, $\beta = 0$, $t = r$ and $\Gamma_i = B_i$.\\

Assume that we have proved the claim for some $\mathcal{F}$. Let $C = C_{\mathcal{F}}$ and $D = D_{\mathcal{F}}$. Let $\beta, \Gamma_{[t]}$ be as above for the choice $\mathcal{F}$ for the down-set. Thus $s, t \leq C \log^D \epsilon^{-1} r^D$. Let $I_0$ be a maximal set in $\mathcal{F}$. The set $I_0$ is non-empty, since we are done otherwise. Let $\mathcal{F}' = \mathcal{F} \setminus \{I_0\}$. Write each $\Gamma_i(x_{[k-1]}) = \Gamma'_i(x_{[k-1]}) + \Psi_i(x_{I_0})$, for a $\mathcal{F}'$-supported multiaffine map $\Gamma'_i$ and a multilinear map $\Psi_i \colon G_{I_0} \to H$. Look at maximal independent set $\lambda^1, \dots, \lambda^d \in \mathbb{F}^t$ such that for each $i \in [d]$, $\ex_{x_{I_0}, h} \chi\Big(\Big(\sum_{j \in [t]} \lambda^i_j \Psi_j(x_{I_0})\Big) \cdot h\Big) \geq \nu = \epsilon 2^{-k} |\mathbb{F}|^{-s}$. For each $i \in [d]$, we apply Theorem~\ref{strongInvThm} to the multilinear map on $G_{I_0} \times H$ given by $(x_{I_0}, h) \mapsto \big(\sum_{j \in [t]} \lambda^i_j \Psi_j(x_{I_0})\big) \cdot h$, to find $r_i \leq \overline{C} \log^{\overline{D}} \nu^{-1}$, multilinear maps $\gamma^i_j \colon G_{J_{ij}} \to \mathbb{F}$ and $\Lambda^i_j \colon G_{I_0 \setminus J_{ij}} \to H$, for $j \in [r_i]$, where $\emptyset \not= J_{ij} \subset I_0$, such that 
\[\Big(\forall x_{I_0} \in G_{I_0}, h \in H\Big)\hspace{6pt}\Big(\sum_{j \in [t]} \lambda^i_j \Psi_j(x_{I_0})\Big) \cdot h = \sum_{j \in [r_i]} \gamma^i_j(x_{J_{ij}}) \Lambda^i_j(x_{I_0 \setminus J_{ij}}) \cdot h.\]
Define a multiaffine map $\beta' \colon G_{[k-1]} \to \mathbb{F}^s \times \mathbb{F}^{\{(i,j) \colon i \in [d], j \in [r_i]\}}$ by 
\[\beta'(x_{[k-1]}) = \Big(\beta(x_{[k-1]}), (\gamma^i_j(x_{J_{ij}}) \colon i \in [d], j \in [r_i])\Big).\]
We now show that $\beta'$ and $\Gamma'_{[t]}, \Lambda^i_j$, $i \in [d], j \in [r_i]$ have the desired properties for the down-set $\mathcal{F}'$.\\

Take arbitrary $\lambda \in \mathbb{F}^r$ and consider $Z_\lambda$. By induction hypothesis, there are distinct layers $L_1, \dots, L_m$ of $\beta$ and multiaffine maps $A_1, \dots, A_m \in \spn \{\Gamma_{[t]}\}$, such that $Z_\lambda \cap L_i = \{x_{[k-1]} \in G_{[k-1]} \colon A_i(x_{[k-1]}) = 0\} \cap L_i$, and $\Big|Z_\lambda \setminus \Big(\cup_{i \in [m]} (Z_\lambda \cap L_i)\Big)\Big| \leq \frac{2^k - |\mathcal{F}|}{2^k} \epsilon |G_{[k-1]}|$. Here $m \leq |\mathbb{F}|^s$. For each $i \in [m]$, write $A_i(x_{[k-1]}) = A'_i(x_{[k-1]}) + C_i(x_{I_0})$ so that $A'_i \in \spn \{\Gamma'_{[t]}\}$ and $C_i \in \spn \{\Psi_{[t]}\}$. Thus, either $\ex_{x_{I_0} \in G_{I_0}, h \in H} \chi(C_i(x_{I_0}) \cdot h) < \nu$, or $C_i$ is a linear combination of $\lambda^1 \cdot \Psi, \dots, \lambda^d \cdot \Psi$.\\
\indent In the former case, by Lemma~\ref{biasHomog}, this means that
\[\begin{split}|\{x_{[k-1]} \in G_{[k-1]} \colon A_i(x_{[k-1]}) = 0\}| = &|G_{[k-1] \setminus I_0}| \ex_{y_{[k-1] \setminus I_0}} |\{x_{I_0} \in G_{I_0} \colon A_i(y_{[k-1] \setminus I_0}, x_{I_0}) = 0\}|\\
= &|G_{[k-1]}| \ex_{y_{[k-1] \setminus I_0}} \ex_{x_{I_0}, h} \chi\Big(A_i(y_{[k-1] \setminus I_0}, x_{I_0}) \cdot h\Big)\\
= &|G_{[k-1]}| \ex_{y_{[k-1] \setminus I_0}} \ex_{x_{I_0}, h} \chi\Big(A_i'(y_{[k-1] \setminus I_0}, x_{I_0}) \cdot h\Big) \chi\Big(C_i(x_{I_0}) \cdot h\Big)\\
\leq &|G_{[k-1]}| \ex_{y_{[k-1] \setminus I_0}} \Big|\ex_{x_{I_0}, h} \chi\Big(A_i'(y_{[k-1] \setminus I_0}, x_{I_0}) \cdot h\Big) \chi\Big(C_i(x_{I_0}) \cdot h\Big)\Big|\\
\leq &|G_{[k-1]}| \ex_{y_{[k-1] \setminus I_0}} \ex_{x_{I_0}, h} \chi\Big(C_i(x_{I_0}) \cdot h\Big)\\
\leq &\nu |G_{[k-1]}| \leq \frac{\epsilon}{2^km}|G_{[k-1]}|,\end{split}\]
where $\tau$ is the map defined by $(x_{I_0}, h) \mapsto C_i(x_{I_0}) \cdot h$. We choose to discard these layers. By doing so, we lose at most $2^{-k}\epsilon$ of density in total.\\
\indent In the latter case, we may find $\mu \in \mathbb{F}^d$ such that $C_i(x_{I_0}) = \sum_{j_1 \in [d], j_2 \in [r_i]} \mu_{j_1} \gamma_{j_2}^{j_1}(x_{J_{j_1j_2}}) \Lambda^{j_1}_{j_2}(x_{I_0 \setminus J_{j_1j_2}})$. Partition $L_i$ into layers $M_1, \dots, M_l$ of $\beta'$. On each layer $M_j$, $C_i$ becomes a linear combination of $\Lambda^{j_1}_{j_2}$, and thus $A_i$ becomes a linear combination of $\Gamma'_{j_1}, \Lambda^{j_1}_{j_2}$, finishing the proof.\\

For the bounds, observe the codimension of $\beta'$ is at most
\[s + \sum_{i \in [d]} r_i \leq s + \overline{C} t \Big(s + k + \log \epsilon^{-1}\Big)^{\overline{D}} \leq 2 C^{\overline{D} + 1} \overline{C} (2k)^{\overline{D}} \Big(r \log \epsilon^{-1}\Big)^{(D + 1)(\overline{D} + 1)},\]
and the number of maps $\Gamma'_{j_1}, \Lambda^{j_1}_{j_2}$ is at most
\[t + \sum_{i \in [d]} r_i \leq 2 C^{\overline{D} + 1} \overline{C} (2k)^{\overline{D}} \Big(r \log \epsilon^{-1}\Big)^{(D + 1)(\overline{D} + 1)},\]
as desired.\end{proof}

The theorem follows by applying the proposition for $\mathcal{F} = \{\emptyset\}$. Then each $Z_\lambda \cap L_i = L_i$ is a layer of $\beta$. The constants may be taken to be
\[C^{\bm{inner}}_{k-1} = \Big(2 \overline{C} (2k)^{\overline{D}}\Big)^{(\overline{D} + 1)^{2^k}}\hspace{1cm}\text{and}\hspace{1cm}D^{\bm{inner}}_{k-1} = (2\overline{D} + 2)^{2^k},\]
where $\overline{C} = C^{\bm{strong}}_k$ and $\overline{D} = D^{\bm{strong}}_k$. Hence, if $t$ is a quantity such that $C^{\bm{weak}}_k \leq 2^{k^{2^{t}}}$ and $D^{\bm{weak}}_k \leq 2^{2^{t}}$, using \eqref{strongBounds}, then
\begin{equation}\label{innerBounds}C^{\bm{inner}}_{k-1} \leq 2^{k^{2^{t + O(k)}}}\hspace{1cm}\text{and}\hspace{1cm}D^{\bm{inner}}_{k-1} \leq 2^{2^{t + O(k)}}.\end{equation}
\qed
%\pagebreak
\boldSection{$\bm{Inner}(k-1) \implies \bm{Columns}(k)$}
\noindent\emph{Proof.} We begin the proof by proving the following lemma.

\begin{lemma}\label{solCrit}Let $G$ be a $\mathbb{F}$-vector space. Let $x_1, \dots, x_r \in G$ and let $\lambda_1, \dots, \lambda_r \in \mathbb{F}$. The following are equivalent.
\begin{itemize}
\item[\textbf{(i)}]There is $y \in G$ such that $x_i \cdot y = \lambda_i$ for each $i \in [r]$.
\item[\textbf{(ii)}]$(\forall \mu \in \mathbb{F}^r) \sum_{i \in [r]} \mu_i x_i = 0 \implies \mu \cdot \lambda = 0$. 
\end{itemize}
\end{lemma}

\begin{proof}\textbf{(i)}$\implies$\textbf{(ii):} Suppose that $\sum_{i \in [r]} \mu_i x_i = 0$. Take dot product with $y$.\\
\textbf{(ii)}$\implies$\textbf{(i):} Take a maximal independent set $\{x_{i_1}, \dots, x_{i_s}\}$ among $x_i$. Renaming $x_i$ if necessary, we may assume that this set is $\{x_1, \dots, x_s\}$. For $i \in [s + 1, r]$, we have some $\mu^i \in \mathbb{F}^s$ such that $x_i = \sum_{j \in [s]} \mu^i_j x_j$. By property \textbf{(ii)} we also have $\lambda_i = \sum_{j \in [s]} \mu^i_j \lambda_j$. Since $x_1, \dots, x_s$ are independent, there is $y$ such that $x_i \cdot y = \lambda_i$ for each $i \in [s]$. Property \textbf{(i)} follows since $\lambda_i$ for $i \in [s + 1, r]$ satisfy the identities above.\end{proof}

Write $\alpha_i(x_{[k]}) = \alpha'_i(x_{[k-1]}) + x_k \cdot A_i(x_{[k-1]})$ for multiaffine maps $\alpha'_i \colon G_{[k-1]} \to \mathbb{F}$ and $A_i \colon G_{[k-1]} \to G_k$. For each coset $\Lambda \subset \mathbb{F}^r$ define $V_{\Lambda} = \Big\{x_{[k-1]} \in G_{[k-1]} \colon (\forall \lambda \in \Lambda) (\exists y \in G_k) \alpha(x_{[k-1]}, y) = \lambda\Big\}$. Applying the lemma above, we may rewrite this set as 
\[V_{\Lambda} = \Big\{x_{[k-1]} \in G_{[k-1]} \colon \Big(\forall \lambda \in \Lambda\Big) \Big(\forall \mu \in \mathbb{F}^r \text{ s. t. } \sum_{i \in [r]} \mu_i A_i(x_{[k-1]}) = 0\Big) \sum_{i \in [r]} \mu_i (\lambda_i - \alpha'_i(x_{[k-1]})) = 0\Big\}.\]
Notice that for each $x_{[k-1]} \in G_{[k-1]}$, the set $\{\alpha(x_{[k-1]}, y) \colon y \in G_k\}$ is a coset $\Lambda$ in $\mathbb{F}^r$, and that $|\{y \in G_k \colon \alpha(x_{[k-1]}, y) \in S\}| = \frac{|\Lambda \cap S|}{|\Lambda|}|G_k|$. Thus, $X$ is the union of sets of the form $V_\Lambda \setminus \Big(\cup_{\substack{\Lambda \subsetneq M\subseteq \mathbb{F}^r\\M\text{ coset}}} V_M\Big)$, where $\Lambda$ are cosets in $\mathbb{F}^r$ such that $|\Lambda \cap S| \geq \epsilon |\Lambda|$.\\

Next, for each $M \leq \mathbb{F}^r$, define 
\[W_M = \Big\{x_{[k-1]} \in G_{[k-1]} \colon \Big(\forall \mu \in M\Big) \sum_{i \in [r]} \mu_i A_i(x_{[k-1]}) = 0\Big\}.\]
Thus
\[V_\Lambda = \cup_{M \leq \mathbb{F}^r} \bigg(\Big(W_M \setminus (\cup_{M' > M} W_{M'})\Big) \cap \Big\{x_{[k-1]} \in G_{[k-1]} \colon \Big(\forall \lambda \in \Lambda\Big)\Big(\forall \mu \in M\Big) \mu \cdot \alpha'(x_{[k-1]}) = \mu \cdot \lambda\Big\}\bigg).\]
Let $\Lambda = \lambda^0 + \Lambda^0$, for a subspace $\Lambda^0$. Notice that
\[\begin{split}&\Big(\forall \lambda \in \Lambda\Big)\Big(\forall \mu \in M\Big) \mu \cdot \alpha'(x_{[k-1]}) = \mu \cdot \lambda\\
&\hspace{1cm}\iff \bigg(\Big(\forall \mu \in M\Big) \mu \cdot \alpha'(x_{[k-1]}) = \mu \cdot \lambda^0\bigg)\text{ and }\bigg(\Big(\forall \mu \in M\Big)\Big(\forall \lambda \in \Lambda_0\Big) \mu \cdot \lambda = 0\bigg).\end{split}\]
Hence
\begin{equation}\label{eqSubspaceCombs}V_\Lambda = \cup_{M \leq (\Lambda^0)^\perp} \bigg(\Big(W_M \setminus (\cup_{M' > M} W_{M'})\Big) \cap \Big\{x_{[k-1]} \in G_{[k-1]} \colon \Big(\forall \mu \in M\Big) \mu \cdot \alpha'(x_{[k-1]}) = \mu \cdot \lambda^0\Big\}\bigg).\end{equation}
Apply Theorem~\ref{innerAppThm} to $A_1, \dots, A_r$ to find $s \leq C^{\bm{inner}}_{k-1} \Big((4r^3 + 3r^2)\log \epsilon^{-1}\Big)^{D^{\bm{inner}}_{k-1}}$, a multiaffine map $\beta \colon G_{[k-1]} \to \mathbb{F}^s$ so that for each $\lambda \in \mathbb{F}^r$ the set $\{\lambda \cdot A = 0\}$ can be internally approximated by layers of $\beta$ up to error of at most $\epsilon |\mathbb{F}|^{-(4r^2 + 3r)}$ in density. Notice that, if $M = \langle \mu^1, \dots, \mu^d\rangle$, where $d \leq r$, then 
\[W_M = \cap_{i \in [d]} W_{\langle \mu^i\rangle}.\]
Recall that $W_{\langle \mu^i \rangle} = \Big\{x_{[k-1]} \in G_{[k-1]} \colon \sum_{i \in [r]} \mu_i A_i(x_{[k-1]}) = 0\Big\}$, so we may we internally approximate $W_{\langle \mu^i \rangle}$ by a union $D_i$ of layers of $\beta$ with error of density at most $\epsilon |\mathbb{F}|^{-(4r^2 + 3r)}$. Thus $W_M \supset \cap_{i \in [d]} D_i$ and
\[|W_M \setminus (\cap_{i \in [d]} D_i)| = |\cup_{i \in [d]} (W_M \setminus D_i)| \leq |\cup_{i \in [d]} (W_{\langle \mu^i \rangle} \setminus D_i)| \leq |\mathbb{F}|^{-4r^2-2r}\epsilon|G_{[k-1]}|.\]
Also, apply Proposition~\ref{BohrApproxSim} to $A_1, \dots, A_r$ to find multiaffine maps $\gamma_1, \dots, \gamma_r \colon G_{[k-1]} \to \mathbb{F}^t$, where $t \leq (5r^3 + 2r^2) \log \epsilon^{-1}$ such that each $W_{\langle \mu \rangle}$ can be externally approximated by $\{\mu \cdot \gamma = 0\}$ with error of density at most $|\mathbb{F}|^{-5r^2 - 2r}\epsilon$. Thus,
\[|(\cap_{i \in [d]} \{\mu^i \cdot \gamma = 0\}) \setminus W_M| \leq \Big|\cup_{i \in [d]} (\{\mu^i \cdot \gamma = 0\} \setminus W_{\langle \mu^i\rangle})\Big| \leq |\mathbb{F}|^{-5r^2 - r}\epsilon|G_{[k-1]}|.\]
Finally, define a multiaffine map $\phi \colon G_{[k-1]} \to \mathbb{F}^r \times \mathbb{F}^s \times \mathbb{F}^t$ by $\phi = (\alpha', \beta, \gamma)$. Then, each $W_M$ can be approximated both internally and externally using layers of $\phi$ up to error of density at most $|\mathbb{F}|^{-4r^2-2r}\epsilon$. Hence, from~\eqref{eqSubspaceCombs}, every $V_\Lambda$ may be approximated both internally and externally using layers of $\phi$ up to error of density at most $|\mathbb{F}|^{-2r^2 -2r}\epsilon$. Finally, each $V_\Lambda \setminus \Big(\cup_{\substack{\Lambda \subsetneq M\subseteq \mathbb{F}^r\\M\text{ coset}}} V_M\Big)$ may be approximated both internally and externally using layers of $\phi$ up to error of density at most $|\mathbb{F}|^{-r^2-r}\epsilon$. Thus, the set  
\[\cup_{\substack{\Lambda\text{ coset}\\ |\Lambda \cap S| \geq \epsilon |\Lambda|}} \bigg(V_\Lambda \setminus \Big(\cup_{\substack{\Lambda \subsetneq M\subseteq \mathbb{F}^r\\M\text{ coset}}} V_M\Big)\bigg)\]
may be approximated both internally and externally using layers of $\phi$ up to error of density at most $\epsilon$, as desired.\\
\indent When it comes to bounds, the codimension of the desired map is $r + s + t$, and we may take $C^{\bm{columns}}_k = 20C^{\bm{inner}}_{k-1}$ and $D^{\bm{columns}}_k = 3D^{\bm{inner}}_{k-1}$. Hence, if $t$ is a quantity such that $C^{\bm{weak}}_k \leq 2^{k^{2^{t}}}$ and $D^{\bm{weak}}_k \leq 2^{2^{t}}$, using \eqref{innerBounds}, then
\begin{equation}\label{columnsBounds}C^{\bm{columns}}_{k} \leq 2^{k^{2^{t + O(k)}}}\hspace{1cm}\text{and}\hspace{1cm}D^{\bm{columns}}_{k} \leq 2^{2^{t + O(k)}}.\end{equation}
\qed
\pagebreak
\boldSection{$\bm{Columns}(k-1) \land \bm{Inner}(k-1) \implies \bm{Conv}(k)$}
\noindent\emph{Proof.} We prove the claim by induction on $k$, followed by induction on $l$. Recall that we use $Z$ in the expressions below also to denote the indicator function of the set $Z$. We include the artificial case $l = 0$. In this case, we have $Z(x_{[k]}) = |\mathbb{F}|^{-r} \sum_{\lambda \in \mathbb{F}^r} \chi\Big(\lambda \cdot \alpha(x_{[k]})\Big)$. Hence, in this case we may take $C^{\bm{conv}}_{k, 0} = 1, D^{\bm{conv}}_{k, 0} = 1$, $c_i = |\mathbb{F}|^{-r}$, $\beta_i = \alpha_i$.\\

Write $C = C^{\bm{conv}}_{k, l-1}, D = D^{\bm{conv}}_{k, l-1}$. By induction hypothesis there are $s, t \leq C \log^{D}_{|\mathbb{F}|}(100 \epsilon^{-2}) r^D$, multiaffine forms $\beta_i \colon G_{[k]} \to \mathbb{F}$ for $i \in [s]$, multiaffine map $\gamma \colon G_{[k] \setminus \{l-1\}} \to \mathbb{F}^t$, constants $c_1, \dots, c_m \in \mathbb{C}$, multiaffine maps $\rho_1, \dots, \rho_m \in \spn\{\beta_{[s]}\}$ and layers $L_1, \dots, L_n$ of $\gamma$ such that
\[\bigg|\bigconv{l-1} \cdots \bigconv{1} Z(x_{[k]}) - \sum_{i \in [m]} c_i \chi\Big(\rho_i(x_{[k]})\Big)\bigg| \leq \frac{\epsilon^2}{100}\]
for all $x_{[k]} \in G_{[k]} \setminus \Big((\cup_{i \in [n]} L_i) \times G_{l-1}\Big)$, $|\cup_{i \in [n]} L_i| \leq \frac{\epsilon^2}{100} |G_{[k] \setminus \{l-1\}}|$ and $\sum_{i \in [m]} |c_i| \leq 1$.\\
\indent Let $E = \Big(\cup_{i \in [n]} L_i\Big) \times G_{l-1}$, which therefore has density at most $\frac{\epsilon^2}{100}$. Write $\rho_i(x_{[k]}) = \rho'_i(x_{[k] \setminus \{l\}}) + \Gamma_i(x_{[k] \setminus \{l\}}) \cdot x_l$ for multiaffine maps $\rho'_i \colon G_{[k] \setminus \{l\}} \to \mathbb{F}$ and $\Gamma_i \colon G_{[k] \setminus \{l\}} \to G_l$. Write $a \overset{\nu}{\approx} b$ if $|a-b| \leq \nu$. Consider $x_{[k]} \in G_{[k]}$ such that $|E_{x_{[k] \setminus \{l\}}}| \leq \frac{\epsilon}{10}|G_{[k] \setminus {l}}|$. For such $x_{[k]}$ we have
\[\begin{split}\bigconv{l} \cdots \bigconv{1} Z(x_{[k]}) = &\ex_{y_l \in G_l} \bigconv{l-1}\cdots \bigconv{1}Z(x_{[l-1]}, y_l + x_l, x_{[l+1, k]}) \overline{\bigconv{l-1}\cdots \bigconv{1}Z(x_{[l-1]}, y_l, x_{[l+1, k]})}\\
\overset{\epsilon/4}{\approx} &\ex_{y_l \in G_l} \sum_{i \in [m]} c_i \hspace{1pt}\chi\Big(\rho'_i(x_{[k] \setminus \{l\}}) + \Gamma_i(x_{[k] \setminus \{l\}}) \cdot (x_l + y_l)\Big) \overline{\bigconv{l-1}\cdots \bigconv{1}Z(x_{[l-1]}, y_l, x_{[l+1, k]})}\end{split}\]
\[\begin{split}\overset{\epsilon/2}{\approx} &\ex_{y_l \in G_l} \sum_{i, j \in [m]} c_i \overline{c_j} \hspace{1pt}\chi\bigg(\Big(\rho'_i - \rho'_j\Big)(x_{[k] \setminus \{l\}}) + \Gamma_i(x_{[k] \setminus \{l\}}) \cdot x_l + \Big(\Gamma_i - \Gamma_j\Big)(x_{[k] \setminus \{l\}}) \cdot y_l\bigg)\\
\phantom{\bigconv{l} \cdots \bigconv{1} Z(x_{[k]})}= &\sum_{i, j \in [m]} c_i \overline{c_j} \hspace{1pt}\chi\bigg(\Big(\rho'_i - \rho'_j\Big)(x_{[k] \setminus \{l\}}) + \Gamma_i(x_{[k] \setminus \{l\}}) \cdot x_l\bigg) \bm{1}\Big((\Gamma_i - \Gamma_j)(x_{[k] \setminus \{l\}}) = 0 \Big).\end{split}\] 
Write $\beta_i(x_{[k]}) = \beta'_i(x_{[k] \setminus \{l\}}) + \Psi_i(x_{[k] \setminus \{l\}}) \cdot x_l$ for multiaffine maps $\beta'_i \colon G_{[k] \setminus \{l\}} \to \mathbb{F}$ and $\Psi_i \colon G_{[k] \setminus \{l\}} \to G_l$. Hence $\Gamma_i \in \spn \{\Psi_{[s]}\}.$ Thus, for each $i,j \in [m]$, there is $\mu^{ij} \in \mathbb{F}^s$ such that $\Gamma_i - \Gamma_j = \sum_{v \in [s]} \mu^{ij}_v \Psi_v$.\\
\indent Apply Proposition~\ref{BohrApproxSim} to maps $\Psi_{[s]}$ to find $u \leq (2s^3 + s) \log (100 \epsilon^{-1})$ and multiaffine maps $\tau_1, \dots,$ $\tau_s \colon G_{[k] \setminus \{l\}} \to \mathbb{F}^u$ such that for each $\lambda \in \mathbb{F}^s$, we have that $\{\sum_{i \in [s]} \lambda_i \tau_i = 0\} \supset \{\sum_{i \in [s]} \lambda_i \Psi_i = 0\}$ and $|\{\sum_{i \in [s]} \lambda_i \tau_i = 0\} \setminus \{\sum_{i \in [s]} \lambda_i \Psi_i = 0\}| \leq |\mathbb{F}|^{-2s^2}\frac{\epsilon}{100} |G_{[k] \setminus \{l\}}|$.\\
\indent Therefore,
\[\begin{split}\bigconv{l} \cdots \bigconv{1} Z(x_{[k]}) = &\ex_{y_l \in G_l} \bigconv{l-1}\cdots \bigconv{1}Z(x_{[l-1]}, y_l + x_l, x_{[l+1, k]}) \overline{\bigconv{l-1}\cdots \bigconv{1}Z(x_{[l-1]}, y_l, x_{[l+1, k]})}\\
\overset{\epsilon/2}{\approx}&\sum_{i, j \in [m]} c_i \overline{c_j} \hspace{1pt}\chi\bigg(\Big(\rho'_i - \rho'_j\Big)(x_{[k] \setminus \{l\}}) + \Gamma_i(x_{[k] \setminus \{l\}}) \cdot x_l\bigg) \bm{1}\Big(\sum_{w \in [s]} \mu^{ij}_w \tau_w(x_{[k] \setminus \{l\}}) = 0 \Big)\\
=&\sum_{i, j \in [m], \nu \in \mathbb{F}^u} c_i \overline{c_j}|\mathbb{F}|^{-u} \hspace{2pt}\chi\bigg(\Big(\rho'_i - \rho'_j\Big)(x_{[k] \setminus \{l\}}) + \Gamma_i(x_{[k] \setminus \{l\}}) \cdot x_l + \sum_{w \in [u]} \mu_w^{ij} \tau_w(x_{[k] \setminus \{l\}}) \cdot \nu\bigg)\end{split}\] 
holds for all $x_{[k]} \in G_{[k]}$ outside the set 
\[\Big\{x_{[k]} \in G_{[k]} \colon |E_{x_{[k] \setminus \{l\}}}| \geq \frac{\epsilon}{10} |G_{[k] \setminus {l}}|\Big\} \cup \Big(\cup_{i,j \in [m]} (\{\mu^{ij} \cdot \tau = 0\} \setminus \{\Gamma_i - \Gamma_j = 0\})\Big) \times G_l.\]
Since $|E| \leq \frac{\epsilon^2}{100}|G_{[k]}|$, we have
\[\Big|\Big\{x_{[k]} \in G_{[k]} \colon |E_{x_{[k] \setminus \{l\}}}| \geq \frac{\epsilon}{10} |G_{[k] \setminus {l}}|\Big\}\Big| \leq \frac{\epsilon}{10}|G_{[k]}|.\]
Also, by the way we chose $\tau$, recalling that $m \leq |\mathbb{F}|^s$, 
\[\Big|\Big(\cup_{i,j \in [m]} (\{\mu^{ij} \cdot \tau = 0\} \setminus \{\Gamma_i - \Gamma_j = 0\})\Big) \times G_l\Big| \leq \frac{\epsilon}{100}|G_{[k]}|.\]

Next, apply Theorem~\ref{innerAppThm} to $\Psi_1, \dots, \Psi_s$, to find a map $\phi \colon G_{[k] \setminus \{l\}} \to \mathbb{F}^{s'}$, where
\[s' \leq C^{\bm{inner}}_{k-1} \Big(2s^2 + s \log (100 \epsilon^{-1}) \Big)^{D^{\bm{inner}}_{k-1}},\]
such that for each $\lambda \in \mathbb{F}^s$, we may approximate $\{\lambda \cdot \Psi = 0\}$ (and hence each $\{\Gamma_i - \Gamma_j = 0\}$) internally using layers of $\phi$ up to error of at most $|\mathbb{F}|^{-2s}\frac{\epsilon}{100}$ in density.\\
\indent Finally, we need to approximate $\Big\{x_{[k]} \in E \colon |E_{x_{[k] \setminus \{l\}}}| \geq \frac{\epsilon}{10}|G_l|\Big\}$ externally by layers of a multiaffine map of bounded codimension up to error of density at most $\frac{\epsilon}{100}$. We apply Theorem~\ref{denseColumnsThm} to $\gamma$, recalling that $E = \Big(\cup_{i \in [n]} L_i\Big) \times G_{l-1}$, to find such a map $\delta \colon G_{[k] \setminus \{l-1, l\}} \to \mathbb{F}^w$, where 
\[w \leq C^{\bm{columns}}_{k-1} \Big(t \log (100 \epsilon^{-1})\Big)^{D^{\bm{columns}}_{k-1}}.\]
Hence, we obtained the desired approximation outside layers $L'_1, \dots, L'_q$ of the multiaffine map $(\tau, \phi, \delta)$ of codimension at most
\[\begin{split}s \cdot u + s' + w \leq &(2s^4 + s^2) \log (100 \epsilon^{-1}) + C^{\bm{inner}}_{k-1} \Big(2s^2 + s \log (100 \epsilon^{-1}) \Big)^{D^{\bm{inner}}_{k-1}}\\
&\hspace{5cm} +  C^{\bm{columns}}_{k-1} \Big(t \log(100 \epsilon^{-1})\Big)^{D^{\bm{columns}}_{k-1}}\\
\leq & \Big(C^{\bm{inner}}_{k-1} 10^{D^{\bm{inner}}_{k-1}} +  C^{\bm{columns}}_{k-1}8^{D^{\bm{columns}}_{k-1}} + 25\Big) \Big(20^DC\Big)^{\max\{4, 2D^{\bm{inner}}_{k-1}, (D + 1) D^{\bm{columns}}_{k-1}\}}\\ &\hspace{5cm}\cdot\Big(r \log \epsilon^{-1}\Big)^{(D + 1) \max\{4, 2D^{\bm{inner}}_{k-1}, (D + 1) D^{\bm{columns}}_{k-1}\}}\end{split}\]
such that $|\cup_{i \in [q]} L'_i | \leq \epsilon |G_{[k]}|$. The multiaffine forms $\beta'_{[s]}, \Big(x_{[k]} \mapsto \Psi_{[s]}(x_{[k] \setminus \{l\}}) \cdot x_l\Big), \tau_{[s], [u]}$ span the set of multiaffine forms used in the arguments of $\chi$ in the approximation sum, and their number is at most
\[s(2 + u) \leq C^4 20^{4D+1}\Big(r \log \epsilon^{-1}\Big)^{4D+1}.\]
Hence, we may take constants
\[C^{\bm{conv}}_{k,l} = \Big(C^{\bm{inner}}_{k-1} 10^{D^{\bm{inner}}_{k-1}} +  C^{\bm{columns}}_{k-1}8^{D^{\bm{columns}}_{k-1}} + 25\Big) \Big(20^DC\Big)^{\max\{4, 2D^{\bm{inner}}_{k-1}, (D + 1) D^{\bm{columns}}_{k-1}\}}\]
and
\[D^{\bm{conv}}_{k,l} = (D + 1) \max\{4, 2D^{\bm{inner}}_{k-1}, (D + 1) D^{\bm{columns}}_{k-1}\},\]
where $C = C^{\bm{conv}}_{k, l-1}, D = D^{\bm{conv}}_{k, l-1}$, which completes the proof. Hence, if $t$ is a quantity such that $C^{\bm{weak}}_k \leq 2^{k^{2^{t}}}$ and $D^{\bm{weak}}_k \leq 2^{2^{t}}$, using \eqref{columnsBounds}, then
\begin{equation}\label{convBounds}C^{\bm{conv}}_{k, k} \leq 2^{k^{2^{t + O(k)}}}\hspace{1cm}\text{and}\hspace{1cm}D^{\bm{conv}}_{k, k} \leq 2^{2^{t + O(k)}}.\end{equation}\qed

\boldSection{$\bm{Conv}(k) \implies \bm{Weak}(k+1)$}
\begin{proof} Let $\alpha \colon G_{[k+1]} \to \mathbb{F}$ be a multilinear form such that $\ex_{x_{[k+1]}} \chi\Big(\alpha(x_{[k+1]})\Big) \geq c$. Let $\alpha(x_{[k+1]}) = A(x_{[k]}) \cdot x_{k+1}$ for a multilinear map $A \colon G_{[k]} \to G_{k+1}$. Then $\{A = 0\}$ has density at least $c$ in $G_{[k]}$. Apply Lemma~\ref{BohrApprox} to find a multilinear $\beta \colon G_{[k]} \to \mathbb{F}^r$, where $r \leq 2^{k} \log c^{-1} + (k +3) \log 2$, such that $\{A = 0\} \subset \{\beta = 0\}$ and $|\{\beta = 0\} \setminus \{A = 0\}| \leq  2^{- k -3}c^{2^k}|G_{[k]}|$. Let $Z = \{\beta = 0\}$. Apply Theorem~\ref{convAppThm} to $\beta$, to find $s, t \leq C^{\bm{conv}}_{k,k} \Big(2^{k}  \log c^{-1} + k + 3\Big)^{2\cdot D^{\bm{conv}}_{k,k}}$, multiaffine forms $\tau_i\colon G_{[k]} \to \mathbb{F}$ for $i \in [s]$, multiaffine map $\gamma \colon G_{[k-1]} \to \mathbb{F}^t$, constants $c_1, \dots, c_m \in \mathbb{C}$, multiaffine maps $\rho_1, \dots, \rho_m \in \spn\{\tau_{[s]}\}$ and layers $L_1, \dots, L_n$ of $\gamma$ such that
\[\bigg| \bigconv{k} \cdots \bigconv{1} Z(x_{[k]}) - \sum_{i \in [m]} c_i \chi\Big(\rho_i(x_{[k]})\Big)\bigg| \leq \frac{1}{8}c^{2^k}\]
for all $x_{[k]} \in G_{[k]} \setminus \Big((\cup_{i \in [n]} L_i) \times G_k\Big)$, $|\cup_{i \in [n]} L_i| \leq \frac{1}{8}c^{2^k} |G_{[k-1]}|$ and $\sum_{i \in [m]} |c_i| \leq 1$.\\
\indent By applying Cauchy-Schwarz inequality several times, we see that 
\[\ex_{x_{[k]}}\bigconv{k} \cdots \bigconv{1} Z(x_{[k]}) = \ex_{x_{[k-1]}} \Big(\ex_{x_k}\bigconv{k-1} \cdots \bigconv{1} Z(x_{[k]})\Big)^2 \geq \Big(\ex_{x_{[k]}}\bigconv{k-1} \cdots \bigconv{1} Z(x_{[k]})\Big)^2 \geq \dots \geq c^{2^k}.\]
By averaging, there is a non-empty layer $D$ of $(\tau, \gamma)$ such that, for each $x_{[k]} \in D$, $\bigconv{k} \cdots \bigconv{1} Z(x_{[k]}) \geq \frac{1}{4}c^{2^k}$.\\

We now give a definition of an \emph{arrangement} of points in $G_{[k]}$. Firstly, we define \emph{$\emptyset$-arrangement} of lengths $l_{[k]} \in G_{[k]}$ to be the singleton sequence whose only element is $l_{[k]}$. For $i \in [k]$, an \emph{$[i]$-arrangement} of lengths $l_{[k]} \in G_{[k]}$ is a sequence of length $2^i$, being a concatenation $(q_1, q_2)$ of two $[i-1]$-arrangements $q_1$ and $q_2$ (for $i = 1$, $[0]$ is taken to be $\emptyset$), where $q_1$ has lengths $(l_{[i-1]}, l_i + y, l_{[i+1,k]})$ and $q_2$ has lengths $(l_{[i-1]}, y, l_{[i+1,k]})$ for some $y \in G_i$. We note the following.
\begin{lemma}
\begin{itemize}
\item[\textbf{(i)}] For a set $S \subset G_{[k]}$, and a given $x_{[k]} \in G_{[k]}$, the number of $[i]$-arrangements of lengths $x_{[k]}$ whose points lie in $S$ is exactly
\[\bigconv{i} \cdots \bigconv{1} S(x_{[k]}) \hspace{2pt}|G_1|^{2^{i-1}} |G_2|^{2^{i-2}} \cdots |G_i|.\]
\item[\textbf{(ii)}] For $x_{[k]}, l_{[k]} \in G_{[k]}$, $j \leq 2^i$, there are exactly $|G_1|^{2^{i-1}-1} |G_2|^{2^{i-2}-1} \cdots |G_i|^{2^0-1}$ $[i]$-arrangements of lengths $l_{[k]}$ that contain $x_{[k]}$ at $j$\textsuperscript{th} position, when $l_{[i + 1, k]} = x_{[i+1,k]}$, and no such $[i]$-arrangements otherwise.
\item[\textbf{(iii)}] If all $2^i$ points of an $[i]$-arrangement of lenghts $x_{[k]}$ lie inside $\{A = 0\}$, then $A(x_{[k]}) = 0$.
\end{itemize}
\end{lemma}

As a few times before, we misuse the notation slightly by writing $S$ for the indicator function of the set $S$.

\begin{proof}\textbf{(i):} We prove the claim by induction on $i$. For $i = 0$, the only $[0]$-arrangement of lengths $x_{[k]}$ is precisely the singleton sequence $(x_{[k]})$. Thus, the number of such arrangements equals $S(x_{[k]})$.\\
\indent Suppose the claim holds for some $i$ and take any $x_{[k]}$. By definition, each $[i+1]$-arrangement of lengths $x_{[k]}$ is concatenation of two $[i]$-arrangements, of lengths $(x_{[i]}, x_i + y, x_{[i+2, k]})$ and $(x_{[i]}, y, x_{[i+2, k]})$, for some $y \in G_{i+1}$. Using this, and the inductive hypothesis, we see that the number of $[i+1]$ arrangements of lengths $x_{[k]}$ is exactly
\[\begin{split}&|G_{i+1}| \ex_{y \in G_{i+1}} \bigconv{i} \dots \bigconv{1} S(x_{[i]}, x_i + y, x_{[i+2, k]}) \bigconv{i} \dots \bigconv{1} S(x_{[i]}, y, x_{[i+2, k]}) \Big(|G_1|^{2^{i-1}} |G_2|^{2^{i-2}} \cdots |G_i|\Big)^2\\
&\hspace{3cm}= \bigconv{i+1} \dots \bigconv{1} S(x_{[k]}) |G_1|^{2^{i}} |G_2|^{2^{i-1}} \cdots |G_{i+1}|.\end{split}\]
\textbf{(ii):} For $i = 0$, this is clear. Suppose the claim holds for $[i]$-arrangments. Assume that $x_{[i + 2, k]} = l_{[i+2, k]}$, otherwise there are clearly no desired arrangements. Note that from part \textbf{(i)} there are exactly $|G_1|^{2^{i-1}} |G_2|^{2^{i-2}} \cdots |G_i|$ of $[i]$-arrangements of any given lengths. If $j \leq 2^i$, then we know that the number of $[i]$-arrangments of lengths $(l_{[i]}, y, l_{[i+2, k]})$ that contain $x_{[k]}$ at $j$\textsuperscript{th} position is $|G_1|^{2^{i-1} - 1}$ $ |G_2|^{2^{i-2} - 1} \cdots$ $|G_i|^{2^0 - 1}$, when $y = x_{i+1}$, and zero otherwise. On the other hand, there are exactly $|G_1|^{2^{i-1}}$ $|G_2|^{2^{i-2}} \cdots$ $|G_i|$ of $[i]$-arrangements of lengths $(l_{[i]}, x_{i+1} - l_{i+1}, l_{[i+2, k]})$. The result now follows. The case $j > 2^i$ can be treated similarly.\\
\textbf{(iii):} For $i = 0$, this is clear. Suppose the claim holds for $[i]$-arrangments. Let $q$ be an $[i+1]$-arrangment of lengths $x_{[k]}$ whose all points lie inside $\{A = 0\}$. Then, $q = (q_1, q_2)$ for an $[i]$-arrangement $q_1$ of lengths $(x_{[i]}, x_i + y, x_{[i+2, k]})$  and an $[i]$-arrangement $q_2$ of lengths $(x_{[i]}, y, x_{[i+2, k]})$. By induction hypothesis, we have $A(x_{[i]}, x_i + y, x_{[i+2, k]}) = 0$ and $A(x_{[i]}, y, x_{[i+2, k]}) =0$. Since $A$ is linear in $(i+1)$\textsuperscript{th} coordinate, $A(x_{[k]}) = 0$, as desired.\end{proof}

By part \textbf{(i)} of the lemma, for each $x_{[k]} \in D$, there are at least $\frac{1}{4}c^{2^k}|G_1|^{2^{k-1}} |G_2|^{2^{k-2}} \cdots |G_k|^{2^0}$ $[k]$-arrangements of lengths $x_{[k]}$ with all points in $\{\beta = 0\}$. Since $|\{\beta = 0\} \setminus \{A = 0\}| \leq 2^{-(k + 3)} c^{2^k} |G_{[k]}|$, and, by \textbf{(ii)}, each point in $G_{[k]}$ belongs to at most $2^k |G_1|^{2^{k-1} - 1} |G_2|^{2^{k-2} - 1} \cdots |G_k|^{0}$ $[k]$-arrangements of lengths $x_{[k]}$ (the factor of $2^k$ comes from the number of positions a point may take in a $[k]$-arrangement),  this implies that for each $x_{[k]} \in D$, at least one $[k]$-arrangement of lengths $x_{[k]}$ has all its points inside $\{A = 0\}$. Using part \textbf{(iii)}, we obtain $A(x_{[d]}) = 0$. Hence, $D \subset \{A = 0\}$. To finish the proof we modify $D$ to a variety inside $\{A=0\}$, but one which is defined by multilinear maps only.

\begin{lemma}Suppose that $A \colon G_{[k]} \to H$ is a multilinear map and that $D$ is a variety of codimension $r$ in $G_{[k]}$ such that $D \subset \{A=0\}$. Then, there is $R \leq 2^{2k} r$, multilinear forms $\beta_i \colon G_{I_i} \to \mathbb{F}$, $\emptyset \not= I_i \subset [k]$ for $i \in [R]$, such that $\{x_{[k]} \in G_{[k]} \colon (\forall i \in [R]) \beta_i(x_{I_i}) = 0\} \subset \{A = 0\}$.\end{lemma}

\begin{proof}By splitting the map that defines $D$ into its multilinear parts, we see that there are mutlilinear forms $\delta_i \colon G_{I_i} \to \mathbb{F}$, $\emptyset \not= I_i \subset [k]$, scalars $\lambda_i \in \mathbb{F}$, $i \leq r' \leq 2^k r$, such that 
\[\emptyset\not=\{x_{[k]} \in G_{[k]} \colon (\forall i \in [r']) \delta_i(x_{I_i}) = \lambda_i\} \subset \{A = 0\}.\]
By induction on $d$, we prove that, misusing the notation above, we may replace the maps by new ones so that $\lambda_i = 0$, when $I_i \cap [d] \not=\emptyset$, at the expense of larger bound $r' \leq 2^{k + d} r$.\\
\indent For $d = 0$, the given maps satisfy these properties. Assume now that the claim holds for some $d \geq 0$. Take an arbitrary point $(v_{[k]})$ such that $\delta_i(v_{I_i}) = \lambda_i$ for all $i  \in [r']$. We claim that if $x_{[k]} \in G_{[k]}$ satisfies $\delta_i(x_{I_i}) = \lambda_i$ for $d + 1 \notin I_i$, $\delta_i(x_{I_i}) = 0$ and $\delta_i(x_{I_i \setminus \{d+1\}}, v_{d + 1}) = \lambda_i$ for $d + 1 \in I_i$, then $A(x_{[k]}) = 0$. Indeed, let $x_{[k]}$ be such a point. It suffices to show that $A(x_{[d]}, x_{d+1} + v_{d+1}, x_{[d+2, k]}) = A(x_{[d]}, v_{d+1}, x_{[d+2, k]}) = 0$. Since $\delta_i(x_{I_i}) = \lambda_i$ for $d+1 \notin I_i$, we just need to show that $\delta_i(x_{I_i \setminus \{d+1\}}; x_{d+1} + v_{d+1}) = \delta_i(x_{I_i \setminus \{d+1\}}; v_{d+1}) = \lambda_i$ for $d +1\in I_i$, which we already know to hold. Note also that $v_{[k]}$ satisfies all these equalities, so we get a non-empty variety. Thus, we may replace the given maps with at most $2r'$ maps with desired properties.\end{proof}

This completes the proof. As far as the bounds are concerned, we may find a variety inside $\{A = 0\}$ with desired properties with codimension bounded by
\[2^{2k}(s + t) \leq 2^{2k+1} C^{\bm{conv}}_{k,k} \Big(2^{k}  \log c^{-1} + k + 3\Big)^{2\cdot D^{\bm{conv}}_{k,k}}.\]
Thus, we may take $C^{\bm{weak}}_{k+1} = 2^{2k+1} C^{\bm{conv}}_{k,k} (2^{k} + k + 3)^{2D^{\bm{conv}}_{k,k}}$ and $D^{\bm{weak}}_{k+1} = 2 D^{\bm{conv}}_{k,k}$.  Hence, if $t$ is a quantity such that $C^{\bm{weak}}_k \leq 2^{k^{2^{t}}}$ and $D^{\bm{weak}}_k \leq 2^{2^{t}}$, using \eqref{convBounds}, then
\begin{equation}\label{weakBounds}C^{\bm{weak}}_{k + 1} \leq 2^{k^{2^{t + O(k)}}}\hspace{1cm}\text{and}\hspace{1cm}D^{\bm{weak}}_{k + 1} \leq 2^{2^{t + O(k)}}.\end{equation}\end{proof}

\end{document}